\documentclass{amsart}

\usepackage[plainpages,urlcolor=red,linktocpage=true]{hyperref}

\usepackage{amscd}
\usepackage{amssymb}
\usepackage{mathrsfs}
\usepackage{amsmath}
\usepackage{mathabx}
\usepackage{amsthm}
\usepackage[all,cmtip]{xy}
\usepackage{enumerate}
\usepackage{enumitem}
\usepackage{ytableau}
\usepackage{tikz}
\usepackage{mathtools}
\usepackage{color}
\usepackage{wasysym}

\allowdisplaybreaks

\usepackage{tikz-cd}
\usetikzlibrary{cd}

\numberwithin{equation}{section}

\newcommand{\thmref}[1]{Theorem~\ref{#1}}
\newcommand{\secref}[1]{Section~\ref{#1}}

\newcommand{\lemref}[1]{Lemma~\ref{#1}}
\newcommand{\propref}[1]{Proposition~\ref{#1}}
\newcommand{\corref}[1]{Corollary~\ref{#1}}
\newcommand{\defref}[1]{Definition~\ref{#1}}

\newcommand{\exref}[1]{Example~\ref{#1}}
\newcommand{\figref}[1]{Figure~\ref{#1}}

\newcommand{\qbinom}[2]{\genfrac{[}{]}{0pt}{}{#1}{#2}}

\newtheorem{theorem}{Theorem}[section]
\newtheorem{lemma}[theorem]{Lemma}
\newtheorem{proposition}[theorem]{Proposition}
\newtheorem{corollary}[theorem]{Corollary}

\theoremstyle{definition}
\newtheorem{example}[theorem]{Example}

\newtheorem{definition}[theorem]{Definition}

\theoremstyle{remark}
\newtheorem{remark}[theorem]{Remark}

\newcommand{\RN}[1]{%
  \textup{\uppercase\expandafter{\romannumeral#1}}%
}

\begin{document}

\title[Explicit generators and relations of the centre]{Explicit generators and relations for \\ the centre of the quantum group}
\author[Y Dai]{Yanmin Dai}
\address[Y Dai]{School of Mathematical Sciences, University of Science and Technology of China, Heifei, 230026, China}
\email{bt2@mail.ustc.edu.cn}

\author[Y Zhang]{Yang Zhang}
\address[Y  Zhang]{School of Mathematics and Statistics, University of Sydney, NSW 2006, Australia}
\email{yang91@mail.ustc.edu.cn}

\keywords{quantum groups, central elements, Harish-Chandra isomorphism}
\subjclass[2010]{17B37, 81R50}

\begin{abstract}
For the standard Drinfeld-Jimbo quantum group  ${\rm U}_q(\mathfrak{g})$ associated with a simple Lie algebra $\mathfrak{g}$, we construct explicit generators of the centre $Z({\rm U}_q(\mathfrak{g}))$, and determine the relations satisfied by the generators.  For $\mathfrak{g}$ of type $A_n(n\geq 2)$, $D_{2k+1}(k\geq 2)$ or $E_6$,  the centre $Z({\rm U}_q(\mathfrak{g}))$ is isomorphic to a quotient of a polynomial algebra in multiple variables, which is described in a uniform manner for all cases.  For $\mathfrak{g}$ of any other type, $Z({\rm U}_q(\mathfrak{g}))$ is generated by $n=$rank$(\mathfrak{g})$ algebraically independent elements.
\end{abstract}
\maketitle

\section{Introduction}

 Let $\mathfrak{g}$ be a finite dimensional simple complex Lie algebra of rank $n$. In the literature \cite{Tan92,Jan96}, there are two different versions of the Drinfeld-Jimbo quantum group \cite{Dri86, Jim85}, which are denoted by ${\rm \overline{U}}_q(\mathfrak{g})$ and ${\rm U}_q(\mathfrak{g})$ respectively with $q$ being an indeterminate. The former contains among generators $K_{\lambda}$ with $\lambda$ in the weight lattice $P$ of $\mathfrak{g}$, while the latter contains those $K_{\alpha}$ with $\alpha$ in the root lattice $Q$ of $\mathfrak{g}$. We will focus on the latter  quantum group ${\rm U}_q(\mathfrak{g})$ (see \secref{sec: def}) and study the structure of its centre.  

Drinfeld \cite{Dri90} and Reshetikhin \cite{Res90} constructed explicitly a natural isomorphism from  the representation ring to the centre of the quantum group ${\rm \overline{U}}_q(\mathfrak{g})$. Their method exploits the quasi-triangular structure of ${\rm \overline{U}}_q(\mathfrak{g})$ and can be generalised to the quantum affine algebras \cite{Eti95}. It turns out that  the centre $Z({\rm \overline{U}}_q(\mathfrak{g}))$ is a polynomial algebra generated by $n$ algebraically independent central elements associated to certain representations.  Algebraically independent explicit generators of $Z({\rm \overline{U}}_q(\mathfrak{g}))$ have been constructed in \cite{Dai}.

In contrast to the case of  ${\rm \overline{U}}_q(\mathfrak{g})$, the centre $Z({\rm U}_q(\mathfrak{g}))$ of  the  quantum group ${\rm U}_q(\mathfrak{g})$ is  not necessarily a polynomial algebra, and much remains to be understood about its algebraic structure.  

A fundamental problem, analogous to the first and second fundamental theorems of classical invariant theory,  is to describe explicit generators of $Z({\rm U}_q(\mathfrak{g}))$ and the relations which they obey. The generators of $Z({\rm U}_q(\mathfrak{g}))$ are elements of ${\rm U}_q(\mathfrak{g})$ which commute with all elements of ${\rm U}_q(\mathfrak{g})$.  They are usually referred to as quantum Casimir operators,  and play important roles in studying symmetries of physical systems.  

We give a complete solution of this problem for all $\mathfrak{g}$ in Theorem \ref{thm: cen}.

Now we briefly describe the key ingredients used in the proof of Theorem \ref{thm: cen}. We point out here that   \cite{LXZ16} proved to be very useful for our study, and will make comments  later on results of {\em op. cit.}. 

Given any finite dimensional ${\rm U}_q(\mathfrak{g})$-module $V$ of type-{\bf 1} with some conditions on the weights, we employ the quasi $R$-matrix of ${\rm U}_q(\mathfrak{g})$ (see e.g., \cite[\S8.3.3]{KS97} for an explicit formula)
to construct an infinite set of explicit central elements $C^{(k)}_{V}$ for $k=1, 2, \dots$ in Definition \ref{def: Gamma} by following a method developed in \cite{ZGB91a,ZGB91b}. Our main theorem (i.e., Theorem \ref{thm: cen}) states that there exists a finite set $\Sigma$ of  ${\rm U}_q(\mathfrak{g})$-modules such that $C_V=C^{(1)}_{V}$ for $V\in \Sigma$ generate the centre $Z({\rm U}_q(\mathfrak{g}))$.

We determine the set $\Sigma$ and obtain the relations satisfied by the generators by making essential use of the quantised Harish-Chandra isomorphism of ${\rm U}_q(\mathfrak{g})$, which is an isomorphism from the centre $Z({\rm U}_q(\mathfrak{g}))$ to the Weyl group $W$ invariant subalgebra $({\rm U}_{{\rm ev}}^0)^W$ \cite{Jan96}, where  ${\rm U}_{{\rm ev}}^0$ is spanned by the even elements $K_{2\lambda}$ for $\lambda\in M:=\frac{1}{2}Q\cap P$ with $\frac{1}{2}Q:=\{\frac{1}{2}\alpha\mid \alpha\in Q\}$  the half root lattice. In particular, we require a specific representation-theoretical description of the isomorphism. 
For this, we consider the Grothendieck algebra $S({\rm U}_q(\mathfrak{g}))$ of the category of finite dimensional ${\rm U}_q(\mathfrak{g})$-modules whose weights are contained in $M$. Then the quantised Harish-Chandra isomorphism leads to an isomorphism from $S({\rm U}_q(\mathfrak{g}))$ to $Z({\rm U}_q(\mathfrak{g}))$, sending each isomorphism class $[V]$ to $C_{V}$.

To gain a conceptual understanding of the Grothendieck algebra $S({\rm U}_q(\mathfrak{g}))$, we bring the monoid algebra $\mathbb{C}[M^+]$ into the picture \cite{LXZ16}, where $M^+:=\frac{1}{2}Q\cap P^+$  denotes the additive monoid consisting of dominant weights in the half root lattice $\frac{1}{2}Q$. We describe the Hilbert basis ${\rm Hilb}(M^+)$, a minimal generating set of $M^+$, and then split the simple Lie algebras into two types (see \eqref{eq: type}). In the case of type $\RN{1}$, the set ${\rm Hilb}(M^+)$ comprises exactly all fundamental weights of $\mathfrak{g}$ by straightforward 
calculation, and hence the associated monoid algebra $\mathbb{C}[M^+]$ is a polynomial algebra. In the case of type $\RN{2}$, where $\mathfrak{g}$ is of $A_n(n\geq 2)$, $D_{2k+1}(k\geq 2)$ or $E_6$, the automorphism of the corresponding Dynkin diagram (see \figref{fig: involution}) induces an involution of the monoid $M^+$, which permits us to describe generators and relations of the monoid algebra $\mathbb{C}[M^+]$ in a unified way.

We prove that there is a natural isomorphism between $S({\rm U}_q(\mathfrak{g}))$ and the monoid algebra $\mathbb{C}(q)[M^+]:=\mathbb{C}(q)\otimes_{\mathbb{C}} \mathbb{C}[M^+]$ over  the field $\mathbb{C}(q)$ of rational functions,  and therefore obtain $Z({\rm U}_q(\mathfrak{g}))\cong S({\rm U}_q(\mathfrak{g}))\cong   \mathbb{C}(q)[M^+]$. By means of these isomorphisms, for each generator of $\mathbb{C}(q)[M^+]$ we construct an explicit generator $C_{T}$ of $Z({\rm U}_q(\mathfrak{g}))$ associated to a certain tensor product $T$ of fundamental representations of $\mathfrak{g}$. Using the presentation of $\mathbb{C}[M^+]$, we determine relations among these generators $C_{T}$.



 
We must point out that the isomorphism $Z({\rm U}_q(\mathfrak{g}))\cong \mathbb{C}(q)[M^+]$ and a presentation of the monoid algebra $\mathbb{C}(q)[M^+]$ were previously obtained in \cite{LXZ16} by a case by case study. Here we have developed a new method for deriving these results, which is conceptual and uniform.  Also the presentation of  $\mathbb{C}(q)[M^+]$ given in \cite{LXZ16} has different (but equivalent) relations from ours  despite the fact that the generating set is the same. 

We should emphasise the difference between the present paper and \cite{LXZ16}.
While we have given explicit generators and relations of the centre $Z({\rm U}_q(\mathfrak{g}))$, the authors of \cite{LXZ16} gave a presentation for the isomorphic algebra $\mathbb{C}(q)[M^+]$ instead. Their results, while being interesting in their own right, 
do not help in constructing explicit generators of $Z({\rm U}_q(\mathfrak{g}))$, which is one of our main concerns in this paper. 

We note that the eigenvalues of higher order central elements $C_{V}^{(k)}$ are computed explicitly in \cite{LZ93,DGL05} for quantum supergroups ${\rm U}_q(\mathfrak{gl}_{m|n})$ and ${\rm U}_q(\mathfrak{osp}_{m|2n})$, where $V$ is the natural representation. With the eigenvalue formula, it is shown in \cite{Li10} that the centre of ${\rm U}_q(\mathfrak{gl}_n)$ is generated by $C_{V}^{(k)}$ for $1\leq k\leq n$. In a sequel to this paper, we will prove an analogue of this result for quantum groups of types $B$, $C$ and $D$.


This paper is organised as follows. In \secref{sec: cencons} we construct an explicit central element $C_V$ from any finite dimensional ${\rm U}_q(\mathfrak{g})$-module $V$ whose weights are contained in $M$, and then state our main theorem  (\thmref{thm: cen}), which will be proved in later sections. In \secref{sec: monoid} we  describe the Hilbert basis ${\rm Hilb}(M^+)$, and give a presentation of the monoid algebra $\mathbb{C}[M^+]$ in  \thmref{thm: monoalg1} and \thmref{thm: monoalg2}, which correspond to the Lie algebras of type $\RN{1}$ and $\RN{2}$ respectively. In \secref{sec: cenalg} we prove our main theorem by showing that the centre $Z({\rm U}_q(\mathfrak{g}))$ is isomorphic to the monoid algebra $\mathbb{C}(q)[M^+]$ from a representation theoretical point of view.

\vspace{0.3cm}
 \noindent
{{\bf Acknowledgements.}} We would like  to thank Professor Ruibin Zhang for advices and help during the course of this work.


\section{Construction of central elements}\label{sec: cencons}

 We first recall the definition of quantum groups.  Given any finite dimensional  ${\rm U}_q(\mathfrak{g})$-module $V$ whose weights are contained in $M$, we construct explicitly  a central element $C_{V}$  by using the quasi $R$-matrix of ${\rm U}_q(\mathfrak{g})$.  Finally, we state the main theorem of this paper.

\subsection{Quantum groups}\label{sec: def}

Let $\mathfrak{g}$ be a finite dimensional simple Lie algebra of rank $n$ over the complex field $\mathbb{C}$. Let $\mathfrak{h}$ be the Cartan subalgebra of $\mathfrak{g}$, and let $\Phi\subset \mathfrak{h}^{\ast}$ be the set of roots. Fix a set $\Phi^+$ of positive roots, and denote by $\Pi=\{\alpha_1,\dots, \alpha_n \}\subseteq \Phi^+$ the set of corresponding simple roots. 

Let $(-,-)$ be a non-degenerate invariant symmetric bilinear form on $\mathfrak{h}^{\ast}$. The Cartan matrix $A=(a_{ij})$ is the $n\times n$ matrix with $a_{ij}=2(\alpha_i,\alpha_j)/(\alpha_i,\alpha_i)$. 
The fundamental weights $\varpi_i$ of $\mathfrak{g}$ are defined by $2(\varpi_i, \alpha_j)/(\alpha_j,\alpha_j)=\delta_{ij}$ for $1\leq i,j\leq n$. We define 
\[
 P=\bigoplus_{i=1}^{n} \mathbb{Z}\varpi_i, \quad Q= \bigoplus_{i=1}^n \mathbb{Z}\alpha_i
\]
to be the weight lattice and root lattice, respectively. Let $P^+ \subset P$  be the set of dominant weights, i.e., weights that are non-negative integer combinations of $\varpi_i$. 

Throughout, let $q$ be an indeterminate and  $\mathbb{C}(q)$ the field of rational functions.  The quantum group ${\rm U}_q(\mathfrak{g})$ \cite{Jan96}  is the unital associative algebra over $\mathbb{C}(q)$ generated by $E_{i},F_{i}$ and  $K_i:=K_{\alpha_i}$ for $1\leq i\leq n$,
 subject to the following relations:
\begin{align*}
K_iK_i^{-1}=1=K_i^{-1}K_i,&\quad K_iK_j=K_jK_i\\
K_iE_{j}K_i^{-1}=&q^{(\alpha_i,\alpha_j)}E_{j}, \\
K_{i}F_{j}K_{i}^{-1}=&q^{-(\alpha_i,\alpha_j)}F_{j}, \\
E_{i}F_{j}-F_{j}E_{i}=&\delta_{ij}\frac{K_{i}-K_{i}^{-1}}{q_i-q_i^{-1}}, \\
\sum\limits_{s=0}^{1-a_{ij}}(-1)^s \qbinom{1-a_{ij}}{s}_{q_i} &
E_{i}^{1-a_{ij}-s}E_{j}E_{i}^{s}=0,\ i\neq j, \\
\sum\limits_{s=0}^{1-a_{ij}}(-1)^s \qbinom{1-a_{ij}}{s}_{q_i}&
F_{i}^{1-a_{ij}-s}F_{j}F_{i}^{s}=0,\ i\neq j,
\end{align*}
where $q_i=q^{(\alpha_i,\alpha_i)/2}$,  and for any $m\in\mathbb{N}$
\begin{equation*}\label{eq: qbinomial}
  [m]_{q_i}=\frac{q_i^{m}-q_i^{-m}}{q_i-q_i^{-1}},
  \quad[m]_{q_i}!=[1]_{q_i}[2]_{q_i}\cdots [m]_{q_i},
  \quad \qbinom{m}{k}_{q_i}=\frac{[m]_{q_i}!}{[m-k]_{q_i}![k]_{q_i}!}.
  \end{equation*}

It is well known that ${\rm U}_q(\mathfrak{g})$ is a Hopf algebra with  co-multiplication $\Delta$, co-unit $\varepsilon$ and antipode $S$ given by
\begin{eqnarray*}
&\Delta(K_i)=K_i{\otimes}K_i, \quad
\Delta(E_{i})=K_{i}{\otimes}E_{i}{+}E_{i}{\otimes}1, \quad
\Delta(F_{i})=F_{i}{\otimes}K_{i}^{-1}{+}1{\otimes}F_i,\\
&\varepsilon (K_i)=1,\quad \varepsilon (E_{i})=0,
\quad \varepsilon({F_{i}})=0,\\
&S(K_i)=K_i^{-1},\quad S(E_{i})={-}K_{i}^{{-}1}E_{i},\quad S(F_{i})={-}F_{i}K_{i}.
\end{eqnarray*}

Write ${\rm U}={\rm U}_q(\mathfrak{g})$. The quantum group  is graded by the root lattice $Q$, i.e., ${\rm U}=\bigoplus_{\nu\in Q} {\rm U}_{\nu}$ with
\begin{equation*}
   {\rm U}_{\nu}=\{u\in{\rm U}\mid  K_{i}uK_i^{-1}=q^{(\nu,\alpha_i)}u, \, \forall i=1, \dots, n\}.
\end{equation*}
Define ${\rm U}^+$ (resp. ${\rm U}^-$) to be the subalgebra generated by all $E_i$ (resp. $F_i$), and introduce ${\rm U}_{\nu}^+={\rm U}_{\nu}\cap {\rm U}^+$ (resp. ${\rm U}_{-\nu}^-={\rm U}_{-\nu}\cap {\rm U}^-$).

The representation theory of ${\rm U}_q(\mathfrak{g})$ is parallel to that of the Lie algebra $\mathfrak{g}$ \cite{Hum72,Jan96}. Throughout, we are concerned with finite dimensional ${\rm U}_q(\mathfrak{g})$-modules of type $1$. Each  ${\rm U}_q(\mathfrak{g})$-module $V$ admits the weight space decomposition $V=\bigoplus_{ \mu\in \Pi(V)} V_{\mu}$, where $V_{\mu}$ is the  weight space of $V$ and  $ \Pi(V)\subset P$ is the set of weights. Define $m_{V}(\mu):={\rm dim} V_{\mu}$. The dominant weights $\lambda\in P^+$ are in bijection with the simple ${\rm U}_q(\mathfrak{g})$-modules $L(\lambda)$ with the highest weight $\lambda$. For any two weights $\lambda,\mu$ we have the partial ordering $\lambda>\mu$ if and only if $\lambda-\mu$  is a sum of positive roots.

\subsection{Central elements}\label{sec: construction}

Let $V$ be an arbitrary finite dimensional ${\rm U}_q(\mathfrak{g})$-module. Let ${\rm Tr}_1$ denote the partial trace on the first tensor factor of ${\rm End}(V)\otimes {\rm U}_q(\mathfrak{g})$, i.e., ${\rm Tr}_1(\xi\otimes x)= {\rm Tr}(\xi)x$ for any $\xi\in {\rm End}(V)$ and $x\in {\rm U}_q(\mathfrak{g})$. Note that ${\rm Tr}_1(\xi\otimes x)$ is an element of ${\rm U}_q(\mathfrak{g})$.

The following crucial lemma is essentially from \cite{ZGB91a}, where the  setting is slightly different from ours. We include a proof in Appendix \ref{sec: proofs}. 

\begin{lemma}\cite[Proposition 1]{ZGB91a}\label{lem: Gamma}
   Given an operator $\Gamma_{V}\in {\rm End}(V)\otimes {\rm U}_q(\mathfrak{g})$ satisfying
   \begin{equation}\label{commrel}
     [\Gamma_V, \Delta(K_i^{\pm 1})]= [\Gamma_V, \Delta(E_i)]= [\Gamma_V, \Delta(F_i)]=0, \quad \forall i,
   \end{equation}
   the elements $C^{(k)}_{V}\in {\rm U}_q(\mathfrak{g})$ for $k=1, 2, \dots$ defined by 
   \begin{equation}\label{eq: cenelmt}
     C^{(k)}_{V}:= {\rm Tr}_1((K_{2\rho} \otimes 1)(\Gamma_V)^k) 
   \end{equation}
   are central in $ {\rm U}_q(\mathfrak{g})$,  where $\rho$ denotes the half sum of positive roots of $\mathfrak{g}$.
\end{lemma}

Using the quasi $R$-matrix of ${\rm U}_q(\mathfrak{g})$, we shall construct  an explicit operator $\Gamma_V$ satisfying \eqref{commrel}.

Recall that the Drinfeld version of quantum group defined over formal power series $\mathbb{C}[[h]]$ admits a universal $R$ matrix \cite{Dri86,ZGB91a}, which is absent  for the  quantum group ${\rm U}_q(\mathfrak{g})$ considered here. But a quasi $R$-matrix $\mathfrak{R}$ exists for ${\rm U}_q(\mathfrak{g})$, which can be described as follows \cite{Lus10}.

Let ${\rm U}_q(\mathfrak{g})\widehat{\otimes} {\rm U}_q(\mathfrak{g})$ be a  completion of the tensor product ${\rm U}_q(\mathfrak{g})\otimes {\rm U}_q(\mathfrak{g})$. There is an algebra automorphism $\phi$ of $ {\rm U}_q(\mathfrak{g})\otimes {\rm U}_q(\mathfrak{g})$ defined by
\[
\begin{aligned}
\phi(K_i\otimes 1)=K_i\otimes1,\quad&\phi(E_i\otimes 1)=E_i\otimes K_{i}^{-1},\quad\phi(F_i\otimes 1)=F_i\otimes K_{i},\\
\phi(1\otimes K_i)=1\otimes K_i,\quad&\phi(1\otimes E_i)=K_{i}^{-1}\otimes E_i ,\quad\phi(1\otimes F_i)=K_{i}\otimes F_i,
\end{aligned}
\]
and $\phi$ can be extended to ${\rm U}_q(\mathfrak{g})\widehat{\otimes} {\rm U}_q(\mathfrak{g})$.  The quasi $R$-matrix $\mathfrak{R}$ is an element of ${\rm U}_q(\mathfrak{g})\widehat{\otimes} {\rm U}_q(\mathfrak{g})$ which has the form 
\[ \mathfrak{R}=1\otimes 1+ \sum_{\nu>0}\Theta_{\nu}\in {\rm U}_q(\mathfrak{g})\widehat{\otimes} {\rm U}_q(\mathfrak{g}),  \]
where $\Theta_{\nu}\in {\rm U}_{-\nu}^-\otimes {\rm U}_{\nu}^+$ for positive $\nu\in Q$. An explicit formula for $\mathfrak{R}$ can be found in, e.g., \cite[\S8.3.3]{KS97} (see \exref{exam: sl2} for case of ${\rm U}_q(\mathfrak{sl}_2)$).  The quasi $R$-matrix satisfies the following relations
\begin{equation}\label{eq: quasiR}
\begin{aligned}
   \mathfrak{R}\Delta(x)&=\phi(\Delta'(x))\mathfrak{R}, \quad
     \mathfrak{R}^T\Delta'(x)&=\phi(\Delta(x))\mathfrak{R}^T, 
\end{aligned}
\end{equation}
where  $\mathfrak{R}^T=T(\mathfrak{R})$ with $T$ being the linear map defined by $T(x\otimes y)= y\otimes x$ for $x,y\in {\rm U}_q(\mathfrak{g})$ (see, e.g.,  \cite[\S 4.3]{Tan92}).


In what follows, we assume that the ${\rm U}_q({\rm g})$-module $V$ satisfies
\begin{equation}\label{eq: weightcon}
   \Pi(V)\subset  M= \frac{1}{2}Q\cap P,
\end{equation}
where $\Pi(V)$ is the set of all weights of $V$.  Denote by $\zeta_V: {\rm U}_q(\mathfrak{g})\rightarrow {\rm GL}(V)$ the linear representation. We define the following elements of ${\rm End}(V) \otimes {\rm U}_q(\mathfrak{g})$:
\begin{equation}\label{eq: R}
	\mathcal{R}_V:=(\zeta_V\otimes {\rm id})(\mathfrak{R}), \quad
\widetilde{\mathcal{R}}^T_V:=(\zeta_V\otimes {\rm id})\phi(\mathfrak{R}^T)
\end{equation}
On the other hand, we define the diagonal part  by 
\begin{equation}\label{eq: KV}
  \mathcal{K}_V:=\sum_{\eta \in \Pi(V)} P_{\eta}\otimes K_{2\eta}, 
\end{equation}
where  $P_{\eta}$ is the linear projection from $V$ to its weight space $V_{\eta}$.  Note that the condition \eqref{eq: weightcon} guarantees that $2\eta\in Q$ for any weight $\eta$ of $V$ and hence  $\mathcal{K}_V\in {\rm End}(V)\otimes {\rm U}_q(\mathfrak{g})$.

\begin{definition}\label{def: Gamma}
Given a finite dimension ${\rm U}_q(\mathfrak{g})$-module $V$ whose weights are contained in $M=\frac{1}{2}Q\cap P$, we define the operator 
\[ \Gamma_V:=\mathcal{K}_V \widetilde{\mathcal{R}}^T_V \mathcal{R}_V \ \ \in {\rm End}(V)\otimes {\rm U}_q(\mathfrak{g}), \]
 where  $\mathcal{R}_V$ and  $\widetilde{\mathcal{R}}^T_V$ are given in \eqref{eq: R} and $\mathcal{K}_V$ is defined by \eqref{eq: KV}. Let 
\[
C^{(k)}_{V}:= {\rm Tr}_1((K_{2\rho} \otimes 1)(\Gamma_V)^k), \quad k=1, 2, \dots,
\]
and write  $C_V=C^{(1)}_V$. 
\end{definition}

By \lemref{lem: Gamma} and Proposition \ref{prop: commrel} below,  $C^{(k)}_V$ are central elements of ${\rm U}_q(\mathfrak{g})$.  

\begin{proposition}\label{prop: commrel}
The element $\Gamma_V$ given in \defref{def: Gamma} satisfies the commutative relations \eqref{commrel}, i.e., $[\Gamma_V, \Delta(x)]=0$ for any $x\in {\rm U}_q(\mathfrak{g})$.
\end{proposition}

A proof of \propref{prop: commrel} is given in Appendix \ref{sec: proofs}.

The following is an example of our construction. 

\begin{example}\label{exam: sl2}
Let $\mathfrak{g}=\mathfrak{sl}_2$.  The quasi $R$-matrix of ${\rm U}_q(\mathfrak{sl}_2)$ is given by 
 $$\mathfrak{R}=\sum_{n=0}^{\infty} q^{\frac{n(n+1)}{2}}\frac{(1-q^{-2})^n}{[n]_q!} F^n\otimes E^n.$$

Let $V$ be the $2$-dimensional simple module, then the weights of $V$ are contained in  $M=\frac{1}{2}Q\cap P=\mathbb{Z}$. 
Denote by $\zeta: {\rm U}_q(\mathfrak{g}) \rightarrow {\rm End}(V)$ the representation corresponding to the standard basis $\{e_1, e_2\}$ of $V$, we have 
\[
\zeta(E)=\begin{pmatrix}0 & 1\\ 0 & 0\end{pmatrix}, \quad 
\zeta(F)=\begin{pmatrix}0 & 0\\ 1& 0\end{pmatrix}, \quad 
\zeta(K)=\begin{pmatrix}q & 0\\ 0 & q^{-1}\end{pmatrix}.
\]
Now 
\begin{align*}
\mathcal{R}_V&=1\otimes 1+(q-q^{-1})\zeta(F)\otimes E, \\
\widetilde{\mathcal{R}}_V^T&=1\otimes 1+(q-q^{-1})\zeta(EK)\otimes K^{-1}F,\\ 
\mathcal{K}_{V}&=P_{1}\otimes K+P_{-1}\otimes K^{-1},
\end{align*}
where $P_{1}$ (resp. $P_{-1}$) is the linear projection from $V$ onto the weight space ${\mathbb C}(q)e_1$ (resp. ${\mathbb C}(q)e_2$), and we have $P_1=\begin{pmatrix}1& 0\\ 0 & 0\end{pmatrix}$ (res. $P_{-1}=\begin{pmatrix}0& 0\\ 0 & 1\end{pmatrix}$).  We have  
\[
\begin{aligned}
\Gamma_V
=&\mathcal{K}_V \widetilde{\mathcal{R}}^T_V\mathcal{R}_V\\
=& P_{1}\otimes K+P_{-1}\otimes K^{-1}+ (q-q^{-1})\zeta(F) \otimes K^{-1}E \\
&+ (1-q^{-2}) \zeta(E)\otimes F + (q-q^{-1})^2q^{-1} P_1\otimes FE.
\end{aligned}
\]
Note that in $\Gamma_V$ the first tensor factors $\zeta(E)$ and $\zeta(F)$ have no contributions to the partial trace ${\rm Tr}_1$.
Using $K_{2\rho}P_{\pm 1}= KP_{\pm 1}= q^{\pm 1}P_{\pm 1}$, we obtain the following central element associated to $V$:
\[
 C_{V}= {\rm Tr}_1((K_{2\rho} \otimes 1) \Gamma_{V})=qK+q^{-1}K^{-1}+(q-q^{-1})^2FE. 
\]
By similar straightforward calculation, one can also express the higher order central elements $C^{(k)}$ as $\mathbb{C}(q)$-linear combinations of the powers $C_V^k$: 
\[ 
\begin{aligned}
  C_V^{(2)}&=q^{-1}C_V^2-q^{-1}-q^{-3},\\
  C_{V}^{(3)}&= q^{-2}C_{V}^3-(2q^{-2}+q^{-4})C_V,\\
  C_V^{(4)}&= q^{-3}C_V^4-(3q^{-3}+q^{-5})C_V^2+q^{-3}+q^{-5}.\\
\end{aligned}
\]
\end{example}

\subsection{The main theorem}
We shall state our main theorem which exhibits explicit generators and relations of the centre of the quantum group. 

Recall from \defref{def: Gamma} that the central elements are associated with ${\rm U}_q(\mathfrak{g})$-modules $V$ whose weights are contained $M$. In particular, the highest weights of these modules are contained in $M^+=\frac{1}{2}Q\cap P^+$, which is an additive monoid, i.e., a commutative semigroup with the identity $0$. An element $x\in M^+$ is said to be irreducible if $x=y+z$ implies either $y=0$  or $z=0$. The Hilbert basis ${\rm Hilb}(M^+)$ of $M^+$ is a minimal set of generators given by its irreducible elements. 


The generators of the centre $Z({\rm U}_q(\mathfrak{g}))$ can be chosen in bijection with the elements of  ${\rm Hilb}(M^+)$.  Given any $\lambda=\sum_{i=1}^na_{i}\varpi_i \in {\rm Hilb}(M^+)$, we define the tensor module
\[ T(\lambda):= \bigotimes_{i=1}^n L(\varpi_i)^{\otimes a_i},   \]
where $L(\varpi_i)$ is the fundamental representation of ${\rm U}_q(\mathfrak{g})$. Note that for any weight $\mu\in \Pi(T(\lambda))$, we have $\lambda-\mu\in Q\subseteq M=\frac{1}{2}Q\cap P$. Therefore, $ \Pi(T(\lambda))\subseteq M$ and we may define the associated central element $C_{T(\lambda)}$. 

Particularly, if $\mathfrak{g}$ is of type $A_1$, $B_n(n\geq 2)$, $C_n(n\geq 3)$, $D_{2k}(k\geq 2)$, $E_7$, $E_8$, $F_4$, and $G_2$, we will show that the Hilbert basis ${\rm Hilb}(M^+)$ consists of all fundamental weights of $\mathfrak{g}$. Hence we obtain $n$ generators $C_{L(\varpi_i)}$ of  $Z({\rm U}_q(\mathfrak{g}))$, which will be shown to be algebraically independent.

To describe relations among central elements $C_{T(\lambda)}$ for $\mathfrak{g}$ of  one of the remaining types $A_n(n\geq 2)$, $D_{2k+1}(k\geq 2)$ and $E_6$, we introduce the automorphism $\sigma$ of the corresponding Dynkin diagram. This is depicted as in \figref{fig: involution}, where each pair of vertices which are connected by a curved double arrow means they are swapped by the involution $\sigma$ and the rest vertices are fixed by $\sigma$. For instance, $\sigma(i)=n+1-i, 1\leq i\leq n$ for type $A_n(n\geq 2)$.

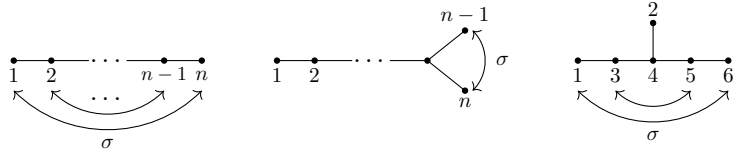
\begin{figure}[h]
   \begin{tikzpicture}

   \draw (0,0)--(0.5,0);
   \filldraw (0,0) circle (1pt);
   \filldraw (0.5,0) circle (1pt);
   \draw (0.5,0)--(1,0);
   \node [scale=1] at (1.25,0) {$\dots$};
   \draw (1.5,0)--(2.5,0);
   \filldraw (2.0,0) circle (1pt);
   \filldraw (2.5,0) circle (1pt);

    \node [scale=0.8,below] at (0,0) {$1$};
    \node [scale=0.8,below] at (0.5,0) {$2$};
    \node [scale=0.75,below] at (2.0,0) {$n-1$};
    \node [scale=0.8,below] at (2.5,-0.04) {$n$};

   \draw[<->] (0,-0.4) to [bend right=45] (2.5,-0.4);
   \draw[<->] (0.5,-0.4) to [bend right=45] (2,-0.4);
   \node [scale=1] at (1.25,-0.5) {$\dots$};
   \node [scale=0.8] at (1.25,-1.1) {$\sigma$};

   \draw (3.5,0)--(4.5,0);
   \filldraw (3.5,0) circle (1pt);
   \filldraw (4.0,0) circle (1pt);
   \node [scale=1] at (4.75,0) {$\dots$};
   \draw (5,0)--(5.5,0);
   \draw (5.5,0)--(6,0.4);
   \draw (5.5,0)--(6,-0.4);
   \filldraw (5.5,0) circle (1pt);
   \filldraw (6,0.4) circle (1pt);
   \filldraw (6,-0.4) circle (1pt);

   \node [scale=0.8,below] at (3.5,0) {$1$};
   \node [scale=0.8,below] at (4.0,0) {$2$};
   \node [scale=0.8,above] at (6,0.4) {$n-1$};
   \node [scale=0.8,below] at (6,-0.4) {$n$};

   \draw[<->] (6.1,0.4) to [bend left=45] (6.1,-0.4);
   \node [scale=0.8] at (6.5,0) {$\sigma$};

   \draw (7.5,0)--(9.5,0);
   \foreach \x in {7.5,8,8.5,9,9.5} { \filldraw (\x,0) circle (1pt);}
   \draw (8.5,0)--(8.5, 0.5);
   \filldraw (8.5,0.5) circle (1pt);

   \node [scale=0.8,below] at (7.5,0) {$1$};
   \node [scale=0.8,below] at (8,0) {$3$};
   \node [scale=0.8,below] at (8.5,0) {$4$};
   \node [scale=0.8,below] at (9,0) {$5$};
   \node [scale=0.8,below] at (9.5,0) {$6$};
   \node [scale=0.8,above] at (8.5,0.5) {$2$};

   \draw[<->] (7.5,-0.4) to [bend right=45] (9.5,-0.4);
   \draw[<->] (8,-0.4) to [bend right=45] (9,-0.4);
   \node [scale=0.8] at (8.5,-1) {$\sigma$};

\end{tikzpicture}
\caption{The involutions $\sigma$ of types $A$, $D$ and $E_6$.}
 \label{fig: involution}
\end{figure}

The automorphism $\sigma$ induces an involution $\sigma_{M^+}$ of the monoid $M^+$. Precisely, if $\lambda=\sum_{i=1}^na_{i} \varpi_i\in {\rm Hilb}(M^+)$, then we define $\overline{\lambda}:=\sigma_{M^+}(\lambda)=\sum_{i=1}^na_{\sigma(i)} \varpi_i$; refer to \lemref{lem: symmertical}. The element $\lambda$ is said to be self-conjugate if $\lambda=\overline{\lambda}$; otherwise, it is called non-self-conjugate. These elements will be characterised explicitly in \lemref{lem: HilMsigma} and \lemref{lem: HilM}.

The following is our main theorem of this paper.

\begin{theorem}\label{thm: cen}
Let $\mathfrak{g}$ be a complex simple Lie algebra of rank $n$, and let ${\rm Hilb}(M^+)$ be the Hilbert basis of the monoid $M^+=\frac{1}{2}Q\cap P^+$, where $\frac{1}{2}Q$ denotes the half root lattice and $P^+$ is the monoid of dominant weights of $\mathfrak{g}$. 
\begin{enumerate}
   \item If $\mathfrak{g}$ is one of the types $A_1$, $B_n(n\geq 2)$, $C_n(n\geq 3)$, $D_{2k}(k\geq 2)$, $E_7$, $E_8$, $F_4$, and $G_2$, then the centre $Z({\rm U}_q(\mathfrak{g}))$ of the quantum group ${\rm U}_q(\mathfrak{g})$ is generated by $n$  algebraically independent elements $C_{L(\varpi_1)}, \dots, C_{L(\varpi_n)}$, where $L(\varpi_i)$ are simple modules corresponding to the fundamental weights $\varpi_i$;

   \item If $\mathfrak{g}$ is one of the types $A_n(n\geq 2)$, $D_{2k+1}(k\geq 2)$ and $E_6$, then the centre $Z({\rm U}_q(\mathfrak{g}))$ of the quantum group ${\rm U}_q(\mathfrak{g})$ is generated by $C_{T(\lambda)}, \lambda\in {\rm Hilb}(M^+)$, subject to the following relations:
    \[
   \begin{aligned}
      &C_{T(\lambda)}C_{T(\bar{\lambda})}=\prod_{i: i<\sigma(i)} C_{T({\mu_i})}^{{\rm max}\{a_i, a_{\sigma(i)}\}},   \\
      &C_{T(\lambda)}^{\ell(\lambda)} = \prod_{i=1}^n C_{T(\nu_i)}^{\ell(\lambda)a_i/s_i} \,\,  \text{with $\lambda\neq  \nu_i, 1\leq i\leq n$} 
   \end{aligned}
   \]
for each non-self-conjugate pair $\{\lambda, \overline{\lambda}\}$ of ${\rm Hilb}(M^+)$ with $\lambda= \sum_{i=1}^na_i\varpi_i$ and $\overline{\lambda}= \sum_{i=1}^na_{\sigma(i)}\varpi_i$, where $\sigma$ is the involution of the Dynkin diagram given by \figref{fig: involution},  $\mu_i\in {\rm Hilb}(M^+)$ are self-conjugate elements  given by \lemref{lem: HilMsigma},  $\nu_i=s_i\varpi_i \in {\rm Hilb}(M^+)$ are scalar multiples of the fundamental weights $\varpi_i$ with $s_i$ given by \lemref{lem: seqint}, and  $\ell(\lambda)$ is a positive integer defined by \eqref{eq: mla}. 

\end{enumerate}
\end{theorem}

The remainder of the paper is on the proof of \thmref{thm: cen}. The main idea of the proof is to show that the centre $Z({\rm U}_q(\mathfrak{g}))$ is isomorphic to the monoid algebra $\mathbb{C}(q)[M^+]$, which is more conceptual and will be studied systematically in \secref{sec: monoid}. The actual proof of \thmref{thm: cen} is given in Section \ref{sec: cenalg}.

\section{The monoid $M^+$}\label{sec: monoid}


In this section we describe the Hilbert basis ${\rm Hilb}(M^+)$ for the monoid  $M^+$ associated to a simple Lie algebra $\mathfrak{g}$.
Using the automorphism of the Dynkin diagram, we  give a presentation of the monoid algebra of $M^+$.

\subsection{The Hilbert basis of $M^+$}
Some results in this subsection can be found in \cite{LXZ16}. We include proofs for them to make the paper more accessible. 

In the sequel, for explicit formulae of fundamental weights of $\mathfrak{g}$ we refer to \cite[\S 13.2, Table 1]{Hum72}, and for Dynkin diagrams we refer to \cite[\S11, Theorem 11.4]{Hum72}.

\begin{lemma}\cite[Lemma 3.4]{LXZ16}\label{lem: finiteness}
For each simple Lie algebra $\mathfrak{g}$, the Hilbert basis ${\rm Hilb}(M^+)$ is finite.
\end{lemma}
\begin{proof}

Observe from  \cite[\S 13.2, Table 1]{Hum72} that for every fundamental weight $\varpi_i$, there exists a minimal positive integer $s_i$ such that $s_i\varpi_i\in M^+$. Given any $\lambda=\sum_{i=1}^na_i\varpi_i\in {\rm Hilb}(M^+)$, we just need to prove  that $0\leq a_i\leq s_i$ for all $i=1,\dots, n$.

Assuming for contradiction that there exists an index $i_0$ such that $a_{i_0}> m_{i_0}$, we have
 $$\lambda=\sum_{i\neq i_0} a_i\varpi_i+(a_{i_0}-s_{i_0})\varpi_{i_0}+s_{i_0}\varpi_{i_0}.$$ 
 Let $\mu=\sum_{i=1,,i\neq i_0}^na_i\varpi_i+(a_{i_0}-s_{i_0})\varpi_{i_0}$.  Then we have  $\mu\in P^+$. As  $\lambda\in\frac{1}{2}Q$ and $s_{i_0}\varpi_{i_0}\in\frac{1}{2}Q$, we also have  $\mu=\lambda- s_{i_0}\varpi_{i_0} \in \frac{1}{2}Q$ and hence $\mu\in M^+= \frac{1}{2}Q\cap P^+$. It follows that $\lambda$ is a sum of two nonzero elements  $\mu$ and $s_{i_0}\varpi_{i_0}$ of $M^+$, contrary to the irreducibility of $\lambda$.
\end{proof}

\begin{lemma}\cite[Lemma3.5]{LXZ16}\label{lem: Hilbtype}
 Let $\varpi_i$ be the fundamental weights of $\mathfrak{g}$.
\begin{enumerate}
 \item If  $\mathfrak{g}$ is of one of types $A_1$, $B_n(n\geq 2)$, $C_n(n\geq 3)$, $D_{2k}(k\geq 2)$, $E_7$, $E_8$, $F_4$, and $G_2$, we have  
\begin{align*}
{\rm Hilb}(M^+)=\{\varpi_1,\cdots,\varpi_n\}.
\end{align*}
 
 \item If  $\mathfrak{g}$ is of type $D_{2k+1}(k\geq 2)$, we have  
\begin{align*}\label{eqn: psimin D odd}
{\rm Hilb}(M^+)=\{\varpi_1,\cdots,\varpi_{n-2},2\varpi_{n-1},2\varpi_n,\varpi_{n-1}+\varpi_n\},
\end{align*}  
where $n=2k+1$.
 \item If $\mathfrak{g}$ is of type $E_6$, we have
 \begin{equation*}\label{eqn: psimin E6}\begin{split}
{\rm Hilb}(M^+)=&\{3\varpi_1,\varpi_2,3\varpi_3,\varpi_4,3\varpi_5,3\varpi_6,\varpi_1+\varpi_3,\varpi_1+\varpi_6,
\varpi_3+\varpi_5,\\
&\,\, \, \varpi_5+\varpi_6,\varpi_1+2\varpi_5,2\varpi_1+\varpi_5,\varpi_3+2\varpi_6,2\varpi_3+\varpi_6\}.
\end{split}
\end{equation*}
\end{enumerate}

\end{lemma}
\begin{proof}
Part (1) can be checked case by case by using \cite[\S 13.2, Table 1]{Hum72}. One has $M^+= \frac{1}{2}Q \cap P^+= P^+ $ in these cases, and  hence ${\rm Hilb}(M^+)$ consists of all  fundamental weights. 

For part (2),  let  $s_i$ be  the smallest positive integer such that $s_i\varpi_i\in\frac{1}{2}Q$ for $1\leq i\leq n$. Using \cite[\S 13.2, Table 1]{Hum72}, it is easy to see that 
\[ s_i=1,\, 1\leq i\leq n-2, \quad \text{and $s_{n-1}=s_n=2$,} \]
where $n\geq 3$ is an odd integer. It follows that $\varpi_i\in {\rm Hilb}(M^+)$ for $1\leq i\leq n-2$ and also $2\varpi_{n-1}, 2\varpi_{n}\in {\rm Hilb}(M^+)$. 

Next we consider the irreducible element of ${\rm Hilb}(M^+)$ which can be written as a sum of fundamental weights. Assume that $\lambda=\sum_{i=1}^na_i\varpi_i \in {\rm Hilb}(M^+)$  is not a multiple of some fundamental weight. Then by the proof of \lemref{lem: finiteness},  we have $a_i\leq s_i$ for each $i$. It is verified readily that $\varpi_{n-1}+ \varpi_{n}\in \frac{1}{2}Q$, while 
$\varpi_i+ \varpi_j\not\in \frac{1}{2} Q$ for any $i\in \{1,\dots, n-2\}$ and $j\in \{n-1, n\}$. Therefore, the only irreducible element  which is a sum of fundamental weights is $\varpi_{n-1}+ \varpi_{n}$. This completes the proof. Part (3) can be done similarly with a suitable computer program
\end{proof}

Now we consider the case of type $A_n$ for  $n\geq 2$. The following lemma is useful.

\begin{lemma}\cite[Lemma 4.3]{LXZ16} \label{lem: McriteriaA}
   Let $\lambda=\sum_{i=1}^n a_i \varpi_i\in P^+$ be a dominant weight of the Lie algebra of type $A_n$. Then $\lambda\in M^+$ if and only if 
   $ \sum_{i=1}^n ia_i \in r_{n+1} \mathbb{Z}$,
   where $r_{n+1}= \frac{n+1}{{\rm gcd}(n+1,2)}$.
\end{lemma}
\begin{proof}
   We need to show that $\lambda \in \frac{1}{2}Q$ if and only if the given condition holds.
   Note that in the case of type $A$ we have $\varpi_i=i\varpi_{1}-(\alpha_{i-1}+2\alpha_{i-2}+ \dots (i-1)\alpha_{i-1})$ for $2\leq i\leq n$. Then we have the expression
   \[
    \begin{aligned}
    \lambda&= a_1\varpi_1 + \sum_{i=2}^n a_i(i\varpi_{1}-(\alpha_{i-1}+2\alpha_{i-2}+ \dots (i-1)\alpha_{i-1}))\\
     &= (\sum_{i=1}^nia_i) \varpi_1 - \sum_{i=2}^n a_i(\alpha_{i-1}+2\alpha_{i-2}+ \dots (i-1)\alpha_{i-1}).
    \end{aligned}
     \]
    It follows that $\lambda\in \frac{1}{2}Q$ if and only if $(\sum_{i=1}^nia_i) \varpi_1 \in \frac{1}{2}Q$. Recall that $\varpi_1= \frac{n}{n+1}\alpha_1+\frac{1}{n+1}(\alpha_n+ 2\alpha_{n-1}+ \dots (n-1)\alpha_2)$. We have $(\sum_{i=1}^nia_i) \varpi_1 \in \frac{1}{2}Q$ if and only if $r_{n+1}= \frac{n+1}{{\rm gcd}(n+1,2)}$ divides  $ \sum_{i=1}^nia_i$.
\end{proof}

\begin{corollary}
  In the case of type $A_n$, assume that $\lambda=\sum_{i=1}^n a_i \varpi_i\in {\rm Hil}(M^+)$. Then we have $\sum_{i=1}^n a_i\leq r_{n+1}$, where $r_{n+1}= \frac{n+1}{{\rm gcd}(n+1,2)}$.
\end{corollary}
\begin{proof}
 For convenience, we denote $\lambda$ by $(a_1, a_2, \dots, a_n)\in \mathbb{N}^n$, where $\mathbb{N}$ is the set of non-negative integers. Let $\prec$ be the lexicographical order on $\mathbb{N}^n$. Without loss of generality, we may assume $a_1\neq 0$. Then there is a strictly increasing sequence:
 \[ \lambda_1\prec \lambda_2\prec \dots \prec \lambda_p=\lambda,  \]
 where $p=\sum_{i=1}^n a_i$, $\lambda_1=(1, 0, \dots, 0)$, and $\lambda_{i+1}-\lambda_i=\varpi_j$ for some $j\geq i$. For any $\lambda_i=(a_{i1}, \dots, a_{in})$, we define 
 \[ b_i\equiv \sum_{j=1}^n j a_{ij} \mod r_{n+1}, \quad 1\leq i\leq p. \]
 Assume for contradiction that $p>r_{n+1}$. Then by the pigeonhole principle, there exists a pair $i_1<i_2$ such that $b_{i_1}=b_{i_2}$, which implies that $r_{n+1}$ divides $\sum_{j=1}^nj(a_{i_2j}-a_{i_1j})$. It follows from  \lemref{lem: McriteriaA} that $\lambda_{i_2}-\lambda_{i_1}\in M^+$, and hence $\lambda= (\lambda_{i_2}-\lambda_{i_1})+(\lambda-\lambda_{i_2}+\lambda_{i_1})$, contrary to the irreducibility of $\lambda$. 
\end{proof}

In accord with the previous lemmas, we split  simple Lie algebras into the following two types:
\begin{equation}\label{eq: type}
   \begin{aligned}
      &\text{Type \RN{1}: $A_1$, $B_n(n\geq 2)$, $C_n(n\geq 3)$, $D_{2k}(k\geq 2)$, $E_7$, $E_8$, $F_4$, and $G_2$},\\ 
      &\text{Type \RN{2}: $A_n(n\geq 2)$, $D_{2k+1}(k\geq 2)$ and $E_6$}.
   \end{aligned}
\end{equation}
For each Lie algebra $\mathfrak{g}$ of type \RN{1}, the Hilbert basis ${\rm Hilb}(M^+)$ comprises exactly all fundamental weights of $\mathfrak{g}$. For  type \RN{2}, we will focus on the symmetry property of $M^+$ derived from the involution of the corresponding Dynkin diagram. This is treated in the following subsection.

\begin{remark}
The classification \eqref{eq: type} is also in accord with the fact that $-1\in W$ if and only if $\mathfrak{g}$ is of type $\RN{1}$ (see, e.g., \cite[Chapter V, \S6.2, Corollary 3]{Bou68}). In this case, $-1$ is the longest element of $W$.
\end{remark}

\subsection{The involution of $M^+$}

For the purpose of this paper, we are only concerned with the involutions corresponding to Lie algebras of type \RN{2}.



\begin{lemma}\label{lem: symmertical}
Let $\mathfrak{g}$ be a simple Lie algebra of type \RN{2}, and let $\sigma$ be the automorphism of the Dynkin diagram given in \figref{fig: involution}.  Then we have the involution of $M^+$:
\begin{equation}\label{eq: involution}
   \sigma_{M^+}: M^+ \rightarrow M^+, \quad \lambda=\sum_{i=1}^n a_i\varpi_i\mapsto \overline{\lambda}=\sum_{i=1}^n a_{\sigma(i)}\varpi_i 
\end{equation}
such that $\sigma_{M^+}^2=1$.
 In particular, if $\lambda\in {\rm Hilb}(M^+)$, then $\overline{\lambda}\in {\rm Hilb}(M^+)$.
\end{lemma}
\begin{proof}
  We only give the proof for type $A$, and  the other cases can be treated similarly. By definition $\sigma$ sends the simple root $\alpha_i$ to $\alpha_{\sigma(i)}=\alpha_{n+1-i}$, and hence $\sigma$ gives rise to an involution $\psi$ of 
  the vector space $\mathbb{Q}Q:=\mathbb{Q} \otimes_{\mathbb{Z}} Q$ over $\mathbb{Q}$. In particular, $\psi$ restricts to an involution of   the half root lattice $\frac{1}{2}Q$.  On the other hand, recall that  the fundamental weights are given by
  \[ \varpi_i=\sum_{j=1}^{i}\frac{(n+1-i)j}{n+1}\alpha_j+\sum_{j=i+1}^{n}\frac{(n+1-j)i}{n+1}\alpha_j\in \mathbb{Q}Q, \quad 1\leq i\leq n.\]
  It is straightforward to check that 
  \[
  \begin{aligned}
   \psi(\varpi_i)&= \sum_{j=1}^{i}\frac{(n+1-i)j}{n+1}\alpha_{n+1-j}+\sum_{j=i+1}^{n}\frac{(n+1-j)i}{n+1}\alpha_{n+1-j}\\
    &= \sum_{k=1}^{n+1-i}\frac{ik}{n+1}\alpha_{k}+\sum_{k=n+2-i}^{n}\frac{(n+1-i)(n+1-k)}{n+1}\alpha_{k},\\
    &= \varpi_{n+i-1}.
  \end{aligned}
  \]
  It follows that $\psi$ induces an involution of the monoid $P^{+}$ with $\psi(\varpi_i)= \varpi_{n+1-i}$ for $1\leq i\leq n$. Therefore, the restriction $\sigma_{M^+}:=\psi|_{M^+}$ is an involution  satisfying
  \[ \sigma_{M^+}(\sum_{i=1}^n a_i\varpi_i)=\sum_{i=1}^n a_i\varpi_{\sigma(i)}=\sum_{i=1}^n  a_{\sigma(i)}\varpi_{i},  \]
  and $\sigma_{M^+}^2=1$. For the last assertion, it is clear that $\lambda$ is irreducible if and only if its image $\overline{\lambda}$ is irreducible.
\end{proof}

\begin{definition}
  The elements $\lambda\in M^+$ satisfying $\overline{\lambda}=\lambda$ are said to be self-conjugate; the other elements of $M^+$ are called non-self-conjugate.
\end{definition}

Now we can split the finite generating set ${\rm Hilb}(M^+)$ into two disjoint subsets, depending on whether they are self-conjugate.
In the following we figure out the self-conjugate elements of ${\rm Hilb}(M^+)$.

\begin{lemma}\label{lem: HilMsigma}
Let $\sigma$ be the involution of the Dynkin diagram given in \figref{fig: involution},  and  let $\lambda\in {\rm Hilb}(M^+)$. Then $ \lambda= \overline{\lambda}$ if and only if $\lambda$ is of the following form:
\begin{enumerate}
  \item $\mu_i= \varpi_i+ \varpi_{\sigma(i)}$ for all $i$ with $i<\sigma(i)$;
  \item $\mu_i= \varpi_i$ for all $i$ with  $\sigma(i)=i$.
\end{enumerate}
\end{lemma}
\begin{proof}

Assume that $\lambda=\sum_{i=1}^n a_i\varpi_i\in {\rm Hilb}(M^+)$ satisfies $\overline{\lambda}= \lambda$. Then we have $a_i=  a_{\sigma(i)}$ for $1\leq i\leq n$. It follows that 
    \[ \lambda= \sum_{i:i<\sigma(i)} a_i ( \varpi_i+\varpi_{\sigma(i)} ) + \sum_{i: \sigma(i)=i} a_i \varpi_{i}. \]
It suffices to show that $\varpi_i+\varpi_{\sigma(i)} \in {\rm Hilb}(M^+)$ whenever $i<\sigma(i)$ and $\varpi_{i}\in {\rm Hilb}(M^+)$ whenever $\sigma(i)=i$.

We only do it for type $A$, and the other two cases can be treated similarly. For $\mathfrak{g}$ of type $A_n$, we first claim that $\varpi_i\notin M^+$ and hence $\varpi_i\notin {\rm Hilb}(M^+)$ whenever $i< \sigma(i)$. By definition (refer to \figref{fig: involution}) $i< \sigma(i)$ if and only if either $n$ is even or $n$ is odd and $i\neq \frac{n+1}{2}$. Using \lemref{lem: McriteriaA}, we obtain that $\varpi_i \notin M^+$ for $1\leq i\leq n$ if $n$ is even, and $\varpi_i \notin M^+$ for $i\neq \frac{n+1}{2} $ if $n$ is odd. Thus our claim follows.  
Secondly, by \lemref{lem: McriteriaA} we have  $ \varpi_i+ \varpi_{\sigma(i)}\in M^+$ for all $i$ with $i< \sigma(i)$. Moreover, by our claim  $ \varpi_i+ \varpi_{\sigma(i)}$ is irreducible and hence $ \varpi_i+ \varpi_{\sigma(i)}\in {\rm Hilb}(M^+)$.  It remains to deal with the case $\sigma(i)=i$. This happens if and only if  $n$ is odd and $i=\frac{n+1}{2}$. In this case, $\varpi_{\frac{n+1}{2}}\in {\rm Hilb}(M^+)$   by \lemref{lem: McriteriaA}. 
\end{proof}

\begin{example}
  Using \lemref{lem: HilMsigma}, we can write out all self-conjugate elements $\mu_i$ of ${\rm Hilb}(M^+)$ explicitly as follows (the indices of fundamental weights $\varpi_i$ are in accord with the labellings of Dynkin diagrams  \figref{fig: involution}).
\begin{enumerate}
   \item Type $A_n (n\geq 2)$. If $n$ is even,  we have
   \[\mu_i:= \varpi_i+ \varpi_{n+1-i},\quad    1\leq  i\leq  \frac{n}{2}.\]
    If $n$ is odd,  we have 
       \[
         \begin{aligned}
          &\mu_i: = \varpi_i+ \varpi_{n+1-i}, \quad  1\leq i \leq \frac{n+1}{2}-1, \\
           &\mu_{\frac{n+1}{2}}:= \varpi_{\frac{n+1}{2}}. 
         \end{aligned}   
       \]
   \item Type $D_{2k+1}(k\geq 2)$.   We have
   \[ \mu_i:= \varpi_i, 1 \leq i\leq n-2, \quad \mu_{n-1}:=\varpi_{n-1}+\varpi_n, \]
   where $n=2k+1$. This can also be verified directly by \lemref{lem: Hilbtype}.

   \item Type $E_6$. We have 
   \[\mu_1:= \varpi_1+\varpi_6,\, \mu_2=\varpi_2,\,  \mu_3:= \varpi_3+\varpi_5,\,  \mu_4=\varpi_4. \] 
\end{enumerate}
\end{example}

For the non-self-conjugate elements of ${\rm Hilb}(M^+)$, we have the following.

\begin{lemma}\label{lem: HilM}
  Let $\lambda=\sum_i^n a_i \varpi_i\in {\rm Hilb}(M^+)$. For each $i=1, \dots, n$, we have
 \begin{enumerate}
   \item If $i\neq \sigma(i)$, then either $a_i a_{\sigma(i)}=0$ or $a_i= a_{\sigma(i)}=1$. In the latter case, all other $a_i$ are $0$, i.e. $\lambda= w_i+ w_{\sigma(i)}$. 
   \item If $i=\sigma(i)$, then either $a_i=0$ or $a_i=1$. In the latter case,  all other $a_i$ are $0$, i.e. $\lambda=\varpi_i$.
 \end{enumerate}
 Therefore, for any non-self-conjugate element $\lambda\in {\rm Hilb}(M^+)$, we have $a_ia_{\sigma(i)}=0$ for $i\neq \sigma(i)$,  and $a_i=0$ for $i=\sigma(i)$.
\end{lemma}
\begin{proof}
  For part (1), we assume that $i\neq \sigma(i)$. If $a_i=a_{\sigma(i)}$, then we may write
   \[ \lambda= a_i(\varpi_i+ \varpi_{\sigma(i)}) + \sum_{j: j\neq i,\sigma(i)} a_j\varpi_j. \]
  It follows from \lemref{lem: HilMsigma} that $\varpi_i+ \varpi_{\sigma(i)}\in {\rm Hilb}(M^+)$. By the irreducibility of $\lambda$, we have $a_i=a_{\sigma(i)}=1$ and all other $a_i$ are $0$, i.e. $\lambda= w_i+ w_{\sigma(i)}$. If $a_i\neq a_{\sigma(i)}$, we need to prove that one of $a_i$ and $a_{\sigma(i)}$ is $0$, that is, $a_i a_{\sigma(i)}=0$. Assume for contradiction that $0<a_i<a_{\sigma(i)}$. Then we may write
  \[
   \lambda=a_i(\varpi_i+ \varpi_{\sigma(i)})+ (a_{\sigma(i)}-a_i) \varpi_{\sigma(i)}+ \sum_{j: j\neq i, \sigma(i)} a_j \varpi_{j}.
  \]
  By the irreducibility of $\lambda$ we have $a_i=a_{\sigma(i)}=1$, which is a contradiction. 

Part (2) can be proved similarly, by using the fact from \lemref{lem: HilMsigma} that $w_i\in {\rm Hilb}(M^+)$ if $i=\sigma(i)$.
\end{proof}


\begin{lemma}\label{lem: rel1}
For any non-self-conjugate element $\lambda= \sum_{i=1}^n a_i\varpi_i \in {\rm Hilb}(M^+)$, we have the relation
\[ \lambda+ \overline{\lambda}= \sum_{i: i<\sigma(i)} {\rm max}\{a_i, a_{\sigma(i)}\}\, \mu_i.  \] 
\end{lemma}
\begin{proof}
  As $\lambda+ \overline{\lambda}$ is fixed by $\sigma_{M^+}$, it can be expressed linearly by elements $\mu_i$  given in  \lemref{lem: HilMsigma}. On the other hand, since $\overline{\lambda}\neq \lambda$ we have $a_i=0$ for $i=\sigma(i)$ and   $a_ia_{\sigma(i)}=0$ for $i\neq \sigma(i)$ by \lemref{lem: HilM}. It follows that 
  \[ \lambda+\overline{\lambda}= \sum_{i: i<\sigma(i)} (a_i+a_{\sigma(i)})\, \mu_i= \sum_{i: i< \sigma(i)} {\rm max}\{a_i, a_{\sigma(i)}\}\, \mu_i,  \]
  where the second equation holds since one of $a_i$ and $a_{\sigma(i)}$ is zero.  
\end{proof}

\begin{example}\label{exam: conj}
 For convenience, we denote $\lambda=\sum_{i}^n a_i \varpi_i$ by $(a_1, a_2, \dots, a_n)$. Then in type $A$ case $\overline{\lambda}=(a_{n}, a_{n-1}, \dots, a_1)$. 
\begin{enumerate}
  \item Type $A_2$. Self-conjugate:  $(1,1)$,  and non-self-conjugate: $(3,0)$ and $(0,3)$.
  \item Type $A_3$. Self-conjugate: $(1,0,1)$, $(0,1,0)$, and non-self-conjugate: $(2,0,0)$, $(0,0,2)$.
  \item Type $A_4$. Self-conjugate: $(1,0,0,1)$, $(0,1,1,0)$, and non-self-conjugate:
  \[ 
   \begin{aligned}
     &(5,0,0,0), (0,5,0,0), (2,0,1,0), (1,2,0,0), (3,1,0,0), (1,0,3,0),\\
     &(0,0,0,5), (0,0,5,0), (0,1,0,2), (0,0,2,1), (0,0,1,3), (0,3,0,1).
    \end{aligned}
\]
\end{enumerate}
\end{example}

We proceed to explore other relations among elements of ${\rm Hilb}(M^+)$. Note that for each fundamental weight $\varpi_i$ there exists a minimal positive integer $s_i$ such that $s_i\varpi_i\in M^+$. Since $s_i$ is minimal, we have  $s_i\varpi_i\in {\rm Hilb}(M^+)$. 
Therefore, we may form a sequence $(s_1, s_2, \dots, s_n)$, which is determined for each Lie algebra of type $\RN{2}$ as follows (note that for type $\RN{1}$ all $s_i$ are equal to $1$). 

\begin{lemma}\label{lem: seqint}
Let $(s_1,s_2, \dots, s_n)$ be a sequence of minimal positive integers such that $s_i\varpi_i\in {\rm Hilb}(M^+)$ for each $i$. Then
\begin{enumerate}
 \item For type $A_n(n\geq 2)$, we have
  \[ s_{i}=s_{n+1-i}=\frac{n+1}{{\rm gcd}(n+1,2i)}, \quad 1\leq i\leq n.\] 
  \item For type $D_{2k+1}(k\geq 2)$, we have 
      \[s_1=\dots=s_{n-1}=1,\quad  s_{n-1}=s_n=2.\]
 \item For type $E_6$, we have 
   \[ s_1=s_3=s_5=s_6=3, \quad s_2=s_4=1.\]     
\end{enumerate}
Recalling that $\sigma$ is the involution of the Dynkin diagram given in \figref{fig: involution},  we have $s_{i}=s_{\sigma(i)}$ for all $i$. In particular, $s_i=s_{\sigma(i)}=1$ if and only if $i=\sigma(i)$.
\end{lemma} 

\begin{proof}
In the case of type $A_n(n\geq 2)$, recall from \lemref{lem: McriteriaA} that $s_i\varpi_i\in M^+$ if and only if $r_{n+1}|is_i$, where $r_{n+1}=(n+1)/{\rm gcd}(n+1,2)$. By the minimality we have $s_i=r_{n+1}/{\rm gcd}(r_{n+1},i)=(n+1)/{\rm gcd}(n+1,2i)$.  The other two cases can be verified directly by \lemref{lem: Hilbtype}. 
\end{proof}

For any $\lambda=\sum_{i=1}^n a_i \varpi_i\in M^+$  we may define
\begin{equation}\label{eq: mla}
  \ell(\lambda):={\rm lcm}\{s_i\mid 1\leq i\leq n\,\,  \text{and}\,\, a_{i}\neq 0 \},
\end{equation}
i.e., $\ell(\lambda)$ is the least common multiple of $s_i$ for all $i$ for which $a_i$ is nonzero. The following lemma is trivial.

\begin{lemma}\label{lem: rel2}
  Let $(s_1, \dots, s_n)$ be the sequence as defined in \lemref{lem: seqint}, and let $\nu_i=s_i\varpi_i$ for $1\leq i\leq n$.
  For any $\lambda=\sum_{i=1}^na_i\varpi_i\in M^+$, we have the relation
  \[ \ell(\lambda)\, \lambda= \sum_{i=1}^{n} \frac{\ell(\lambda)a_i}{s_i} \nu_i, \]
  i.e., $\ell(\lambda)\lambda$ is an integer combination of $\nu_i$.
\end{lemma}

We have obtained relations among generators of ${\rm Hilb}(M^+)$ from \lemref{lem: rel1} and \lemref{lem: rel2}. Next we will show that these relations are enough in the sense that they give rise to complete relations among generators of  the monoid algebra of $M^+$.

\subsection{The monoid algebra}\label{sec: monoidalg}
Given the monoid $M^+$ of any simple Lie algebra, we may form the monoid algebra $\mathbb{C}[M^+]$. As a $\mathbb{C}$-vector space, $\mathbb{C}[M^+]$ has a basis consisting of symbols $X^{\lambda}, \lambda\in M^{+}$, with multiplication given by the bilinear extension of $X^{\lambda}X^{\mu}=X^{\lambda+\mu}$. We agree that $X^0=1$. 

Clearly, $\{X^{\lambda}\mid \lambda\in {\rm Hilb}(M^+)\}$ is a generating set of $\mathbb{C}[M^+]$. In what follows,  we shall determine relations among these generators, and hence obtain a presentation of the monoid algebra $\mathbb{C}[M^+]$.

Recall that for the  simple Lie algebra of type $\RN{1}$ considered in \lemref{lem: Hilbtype}, the Hilbert basis ${\rm Hilb}(M^+)$ consists of all fundamental weights $\varpi_i$ of $\mathfrak{g}$.

\begin{theorem}\label{thm: monoalg1}
  Let $\mathfrak{g}$ be a simple Lie algebra of type $\RN{1}$ with rank $n$, and let $M^+$ be the monoid associated with $\mathfrak{g}$. Then the monoid algebra  $\mathbb{C}[M^+]$ is isomorphic to the polynomial algebra $\mathbb{C}[x_1, \dots, x_n]$ in $n$ variables $x_i$. 
\end{theorem}
\begin{proof}
  Since $\mathfrak{g}$ is of type \RN{1},  the monoid algebra $\mathbb{C}[M^+]$ is generated by $X^{\varpi_i}$ for $1\leq i\leq n$. As the fundamental weights $\varpi_i$ are linearly independent, $\mathbb{C}[M^+]$ is isomorphic to the polynomial algebra $\mathbb{C}[x_1, \dots, x_n]$ in $n$ variables $x_i$, with each $X^{\varpi_i}$ assigned to $x_i$.
\end{proof}

We consider the type $\RN{2}$ case, i.e., $\mathfrak{g}$ is one of the types $A_n(n\geq 2)$, $D_{2k+1}(k\geq 2)$ and $E_6$.  In this case, ${\rm Hilb}(M^+)$ is a disjoint union of self-conjugate elements and non-self-conjugate elements. The non-self-conjugate elements appear in pairs; we use $\{\lambda, \overline{\lambda}\}$ to indicate that $\lambda$ and $ \overline{\lambda}$ are conjugate to each other.

\begin{theorem}\label{thm: monoalg2}
  Let $\mathfrak{g}$ be a simple Lie algebra of type $\RN{2}$ with rank $n$, and let $\sigma$ be the  involution of the corresponding  Dynkin diagram given by \figref{fig: involution}. Let ${\rm Hilb}(M^+)$ be the Hilbert basis of the monoid $M^+$ associated with $\mathfrak{g}$. Then the monoid algebra  $\mathbb{C}[M^+]$ is isomorphic to  $\mathcal{A}=\mathcal{P}/\mathcal{I}$, where $\mathcal{P}$ is the polynomial algebra over $\mathbb{C}$ in variables $x_{\lambda}, \lambda\in {\rm Hilb}(M^+)$, and $\mathcal{I}$ is the ideal of $\mathcal{P}$ generated by 
   \[
   \begin{aligned}
      x_{\lambda}x_{\bar{\lambda}}&- \prod_{i:\, i<\sigma(i)} x_{\mu_i}^{{\rm max}\{a_i, a_{\sigma(i)}\}},   \\
      x_{\lambda}^{\ell(\lambda)}& - \prod_{i=1}^n x_{\nu_i}^{\ell(\lambda)a_i/s_i} \,\,  \text{with $\lambda\neq  \nu_i, 1\leq i\leq n$} 
   \end{aligned}
   \]
for each non-self-conjugate pair $\{\lambda, \overline{\lambda}\}$ of ${\rm Hilb}(M^+)$ with $\lambda= \sum_{i=1}^na_i\varpi_i$ and $\overline{\lambda}= \sum_{i=1}^na_{\sigma(i)}\varpi_i$, where $\mu_i\in {\rm Hilb}(M^+)$ are self-conjugate elements  given by \lemref{lem: HilMsigma},  $\nu_i=s_i\varpi_i \in {\rm Hilb}(M^+)$ are scalar multiples of the fundamental weights $\varpi_i$ with $s_i$ given by \lemref{lem: seqint}, and  $\ell(\lambda)$ is defined by \eqref{eq: mla}. 

\end{theorem}

Before embarking on the proof, let us illustrate this theorem with  examples. 

\begin{example}\label{exam: rel}\
\begin{enumerate} 
  \item Type $A_2$. ${\rm Hilb}(M^+)$ consists of the following elements:
  \[ \mu_1=\varpi_1+\varpi_2,\,\, \{\nu_1=3\varpi_1, \nu_2=3\varpi_2\}. \]
   The ideal $\mathcal{I}$ is generated by  $x_{\nu_1}x_{\nu_2}-x_{\mu_1}^3$. Ideals for type $A_3$ and $A_4$ can be obtained by using \exref{exam: conj}.

  \item Type $D_{2k+1}(k\geq 2)$. ${\rm Hilb}(M^+)$ consists of the following elements:
  \[ 
   \begin{aligned}
       &\mu_i=\nu_i=\varpi_i, \quad 1\leq i\leq n-2,\\
       &\mu_{n-1}=\varpi_{n-1}+\varpi_n, \quad \{\nu_{n-1}=2\varpi_{n-1}, \nu_n=2\varpi_n\},
   \end{aligned}
   \]
   where $n=2k+1$. The ideal $\mathcal{I}$ is generated by $x_{\nu_{n-1}} x_{\nu_n}- x_{\mu_{n-1}}^2$.

  \item Type $E_6$. ${\rm Hilb}(M^+)$ consists of the following elements: 
  \[
   \begin{aligned}
     &\mu_1=\varpi_1+\varpi_6, \,\, \mu_2=\nu_2=\varpi_2, \,\, \mu_3=\varpi_3+\varpi_5,\,\, \mu_4=\nu_4=\varpi_4,\\
     &\{\nu_1=3\varpi_1,\nu_6=3\varpi_6\}, \,\, \{\nu_3=3\varpi_3,\nu_5=3\varpi_5\},\\
     &\{ \varpi_1+\varpi_3, \varpi_5+\varpi_6 \}, \,\, \{ \varpi_1+2\varpi_5, 2\varpi_3+\varpi_6 \},\,\, \{ 2\varpi_1+\varpi_5, \varpi_3+2\varpi_6 \}.
   \end{aligned}
  \]
  Hence the ideal $\mathcal{I}$ is generated by the following binomials:
\[
 \begin{aligned}
    &x_{\nu_1}x_{\nu_6}-x_{\mu_1}^3, \quad x_{\nu_3}x_{\nu_5}-x_{\mu_3}^3, \quad 
    x_{\varpi_1+\varpi_3} x_{\varpi_5+\varpi_6}- x_{\mu_1}x_{\mu_3}, \\
    &x_{\varpi_1+2\varpi_5} x_{2\varpi_3+\varpi_6}- x_{\mu_1}x_{\mu_3}^2,\quad 
      x_{2\varpi_1+\varpi_5} x_{\varpi_3+2\varpi_6 }- x_{\mu_1}^2x_{\mu_3},\\
     & x_{\varpi_1+\varpi_3}^3-x_{\nu_1}x_{\nu_3},\quad 
       x_{\varpi_1+2\varpi_5}^3-x_{\nu_1}x_{\nu_5}^2,\quad 
       x_{2\varpi_1+\varpi_5}^3-x_{\nu_1}^2x_{\nu_5}.\\      
 \end{aligned} 
\]
  \end{enumerate}  
\end{example}

\begin{proof}[Proof of \thmref{thm: monoalg2}]
  We define the following surjective algebra homomorphism 
  $$\varphi: \mathcal{P} \rightarrow \mathbb{C}[M^+], \quad x_{\lambda}\mapsto X^{\lambda}, \,  \lambda\in {\rm Hilb}(M^+).$$ 
  Then by \lemref{lem: rel1} and \lemref{lem: rel2} we have $\mathcal{I}\subseteq {\rm Ker}\, \varphi$. We need to prove that ${\rm Ker}\, \varphi\subseteq \mathcal{I}$, whence $\mathcal{A}=\mathcal{P}/\mathcal{I}$ is isomorphic to $\mathbb{C}[M^+]$. 

We start by defining a polynomial subalgebra $R\subseteq \mathcal{P}$ and show that the restriction $\varphi|_{R}$ is injective. Let $I:=\{i\mid i<\sigma(i) \text{\, for $1\leq i\leq n$} \}$ and 
\[ \Upsilon:=\{\mu_i\mid i\in I \} \cup \{ \nu_i\mid i\in \{1,2, \dots,n\}\backslash I\}.  \]
are linearly independent. Then it can be verified case by case  for all $A_n(n\geq 2)$, $D_{2k+1}(k\geq 2)$ and $E_6$ that $\Upsilon$ is a linearly independent set. Denote by $R:=\mathbb{C}[x_{\lambda}, \lambda\in \Upsilon]$ the polynomial subalgebra of $\mathcal{P}$. By the linear independence of $\Upsilon$, the image $\varphi(R)$ is a polynomial subalgebra of $\mathbb{C}[M^+]$ generated by the algebraically independent elements $X^{\lambda}, \lambda\in \Upsilon$. Therefore, $\varphi|_{R}$ is an isomorphism and $R\cap {\rm Ker}\, \varphi=0$. Define the multiplicatively closed sets $S:=R-\{0\}$ and $\varphi(S)$. 

Next we consider the rings of fractions $S^{-1}\mathcal{P}$ and $\varphi(S)^{-1}\mathbb{C}[M^+]$, and the induced surjective homomorphism 
\[\varphi_S: S^{-1}\mathcal{P} \rightarrow \varphi(S)^{-1}\mathbb{C}[M^+]  \]
given by $\varphi_S(s^{-1}a)= \varphi(s)^{-1}\varphi(a)$ for $a\in \mathcal{P}$ and $s\in S$.  Clearly, ${\rm Ker}\, \varphi\subseteq {\rm Ker}\, \varphi_{S}\cap \mathcal{P}$. We claim that $S^{-1}\mathcal{I}$ is a maximal ideal of $S^{-1}\mathcal{P}$.  If the claim is done, then  ${\rm Ker}\, \varphi_S=S^{-1}\mathcal{I}$ since $S^{-1}\mathcal{I}\subseteq {\rm Ker}\, \varphi_S$. Moreover, $\mathcal{I}$ is a prime ideal since $S^{-1}\mathcal{I}$ is maximal (hence a prime ideal) and $\mathcal{I}\cap S=0$ (follows from $R\cap {\rm Ker}\,\varphi=0$). It follows that  $S^{-1}\mathcal{I}\cap \mathcal{P} =\mathcal{I}$. Combing all above together, we obtain
\[ {\rm Ker}\, \varphi \subseteq {\rm Ker}\, \varphi_S\cap \mathcal{P}=S^{-1}\mathcal{I}\cap \mathcal{P} =\mathcal{I} \]
as required. 

It remains to prove our claim, which is equivalent to showing that $S^{-1}\mathcal{P}/S^{-1}\mathcal{I}\cong S^{-1}\mathcal{A}$ is a field. Consider the fraction field $F=S^{-1}R$ of $R$. We will adjoin extra elements $x_{\lambda}, \lambda\in {\rm Hilb}(M^+)\backslash\Upsilon$ to $F$, and then $S^{-1}\mathcal{A}$ is equal to the resulting extension field of $F$.  Note that by \lemref{lem: HilM} and \lemref{lem: seqint} we have $x_{\mu_i}=x_{\nu_i}=x_{\varpi_i}\in F$ for all $i=\sigma(i)$. For any $i<\sigma(i)$, we define $F_1:=F[x_{\nu_i}]\cong F[t]/(t-x_{\nu_{\sigma(i)}}^{-1}\prod_{i, i<\sigma(i)}x_{\mu_i}^{s_i})$. Since $x_{\nu_{\sigma(i)}}^{-1}\prod_{i, i<\sigma(i)}x_{\mu_i}^{s_i}\in F$, we have $F_1=F$ and hence $x_{\nu_i}\in F$. Therefore, all $x_{\mu_i}$ and $x_{\nu_i}$ belong to the fraction field $F$. 
Now taking an arbitrary non-self-conjugate pair $\{\lambda, \overline{\lambda}\}$ of ${\rm Hilb}(M^+)$ with $\lambda\neq \nu_i$, we define $F_2=F[x_{\lambda}]\cong F[t]/(t^{\ell(\lambda)}- \prod_{i=1}^n x_{\nu_i}^{\ell(\lambda)a_i/s_i})$. Since $x_{\lambda}$ is algebraic over $F$,  $F[x_{\lambda}]=F(x_{\lambda})$ is a field. Similarly, define $F_3=F_2[x_{\bar{\lambda}}]\cong F_2[t]/(t-x_{\lambda}^{-1}\prod_{i, i<\sigma(i)}x_{\mu_i}^{{\rm max}\{a_i,a_{\sigma(i)}\} } )$. Since $x_{\lambda}^{-1}\prod_{i, i<\sigma(i)}x_{\mu_i}^{{\rm max}\{a_i,a_{\sigma(i)}\} }\in F_2$, we have $F_2=F_3$ is a field. Repeating the above step for each non-self-conjugate pair $\{\lambda, \overline{\lambda}\}$, we obtain an extension field $F^{\prime}=S^{-1}\mathcal{A}$ of $F$. This completes the proof.  
\end{proof}


\section{Proof of the main theorem}\label{sec: cenalg}

This section is devoted to proving  \thmref{thm: cen}. We will review the quantised Harish-Chandra theorem, which allows us to construct explicitly an isomorphism between the centre $Z({\rm U}_q(\mathfrak{g}))$ and   the Grothendieck algebra $S({\rm U}_q(\mathfrak{g}))$ of the category of finite dimensional ${\rm U}_q(\mathfrak{g})$-modules  whose weights are contained in $M$.   Then we  show that $S({\rm U}_q(\mathfrak{g}))$ is isomorphic to the monoid algebra $\mathbb{C}(q)[M^+]$ over $\mathbb{C}(q)$, and hence $Z({\rm U}_q(\mathfrak{g}))\cong \mathbb{C}(q)[M^+]$. Combining with the presentation of $\mathbb{C}[M^+]$, we obtain  explicit  generators and relations  of  $Z({\rm U}_q(\mathfrak{g}))$ as given in \thmref{thm: cen}.

\subsection{The Harish-Chandra isomorphism}\label{sec: HCiso}
We will follow \cite[Chapter 6]{Jan96} and retain notation from \secref{sec: def}. Write ${\rm U}={\rm U}_q(\mathfrak{g})$. 
Recall that  the quantum group ${\rm U}$ is graded by the root lattice $Q$, i.e., ${\rm U}=\bigoplus_{\nu\in Q} {\rm U}_{\nu}$.
In particular, ${\rm U}_0= {\rm U}^0\oplus \oplus_{\nu>0} {\rm U}_{-\nu}^- {\rm U}^0 {\rm U}_{\nu}^+  $, where $ {\rm U}^0$ denotes the subalgebra generated by all $K_i^{\pm 1}$. It is known that the projection $$\pi: {\rm U}_0\rightarrow {\rm U}^0$$ is an algebra homomorphism, and  the centre $Z({\rm U}_q(\mathfrak{g}))$   is contained in ${\rm U}_0$. 

The Harish-Chandra isomorphism identifies $Z({\rm U}_q(\mathfrak{g}))\subseteq {\rm U}_0$ with a $W$-invariant subalgebra of ${\rm U}^0$. Precisely, we define an algebra automorphism of ${\rm U}^0$  by
\[ \gamma_{-\rho}: {\rm U}^0\rightarrow {\rm U}^0, \quad K_{\alpha} \mapsto q^{(-\rho, \alpha)} K_{\alpha},  \]
for any $\alpha\in Q$, where $\rho$ denotes the half sum of positive roots of $\mathfrak{g}$.  Then the composite $\gamma_{-\rho}\circ \pi$ is called the  Harish-Chandra homomorphism, under which the image of $Z({\rm U})$  can be described as follows. 

Recall that  $\frac{1}{2}Q:=\{\frac{1}{2} \alpha \mid \alpha\in Q \}$ and $M=\frac{1}{2}Q\cap P$. We define
\[ {\rm U}_{\rm ev}^0:=\langle K_{2\lambda} \mid \lambda \in M\rangle \]
to be the subalgebra of ${\rm U}^0$ spanned by $K_{2\lambda}$ for all $\lambda\in M$. Recall that the Weyl group $W$ of $\mathfrak{g}$ acts naturally on ${\rm U}^0$ via $w.K_{\alpha}= K_{w\alpha}$ for any $w\in W$ and $\alpha\in Q$. This action carries over to ${\rm U}_{{\rm ev}}^0$, and we denote by $({\rm U}_{{\rm ev}}^0)^W$ the $W$-invariant subalgebra.

\begin{theorem}\cite[Theorem 6.25]{Jan96}\label{thm: HCiso}
The Harish-Chandra homomorphism
\begin{equation}\label{eq: HCmap}
\gamma_{-\rho}\circ\pi:Z({\rm U}_q(\mathfrak{g}))\rightarrow({\rm U}_{\rm ev}^0)^W
\end{equation}
is an isomorphism.
\end{theorem}


Note that for each $\lambda\in M=\frac{1}{2}Q\cap P$, there exists a unique $w\in W$ such that $w\lambda\in M^+=\frac{1}{2}Q\cap P^+$.

\begin{lemma}\label{lem: basisUW}
 The elements
${\rm av}(\lambda)=\sum_{w \in W}K_{2w\lambda}, \lambda \in M^+$
form a basis of $({\rm U}_{\rm ev}^0)^W$.
\end{lemma}
\begin{proof}
  Clearly, ${\rm av}(\lambda)\in ({\rm U}_{\rm ev}^0)^W$. If the finite sum $f=\sum_{\lambda} c_{\lambda}K_{2\lambda}$ is $W$-invariant, then 
  $$f= \frac{1}{|W|} \sum_{\lambda}\sum_{w\in W}c_{\lambda}w.K_{2\lambda}=\frac{1}{|W|} \sum_{\lambda} c_{\lambda}(\sum_{w\in W}K_{2w\lambda} ).$$
 Since $\sum_{w\in W}K_{2w\lambda}= {\rm av}(\lambda_0)$ for a unique dominant weight $\lambda_0\in M^+$, the element $f$ can be expressed uniquely as a  linear combination of some ${\rm av}(\lambda), \lambda\in M^+$. Therefore, ${\rm av}(\lambda), \lambda \in M^+$ form a basis for $({\rm U}_{\rm ev}^0)^W$.
\end{proof}

\subsection{Representation theoretical viewpoint}\label{sec: gen}
 
Recall from  \secref{sec: construction} our construction of the central element $C_{V}$. The image of $C_V$ under the Harish-Chandra isomorphism can be calculated as follows.

\begin{lemma}\label{lem: imHC}
  Let $V$ be a finite dimensional ${\rm U}_q(\mathfrak{g})$-module with $\Pi(V)\subseteq M$, and let $C_{V}$ be the associated central element as defined in \defref{def: Gamma}.  Then we have
  \[\gamma_{-\rho}\circ\pi (C_{V})=\sum_{\mu\in \Pi(V)}m_{V}(\mu)K_{2\mu}\in ({\rm U}_{\rm ev}^0)^W,\]
  where $m_{V}(\mu)= {\rm dim} L(\lambda)_{\mu}$.
\end{lemma}
\begin{proof}
  Recall  that $\pi$ is an algebra homomorphism from ${\rm U}_0$ to ${\rm U}^0$, where the latter is a subalgebra generated by all $K_i^{\pm 1}$. Hence we have
  \[ 
  \begin{aligned}
      \pi(C_{V})&=\pi ({\rm Tr}_1((K_{2\rho}\otimes 1) \mathcal{K}_{V}\widetilde{\mathcal{R}}^T_V \mathcal{R}_V  )) \\
      &= {\rm Tr}_1((K_{2\rho}\otimes 1) \mathcal{K}_{V})\\
      &= \sum_{\mu\in\Pi(V)}q^{(\mu,2\rho)}m_{V}(\mu)K_{2\mu}.
  \end{aligned}
  \]
  By the definition of $\gamma_{-\rho}$  we have
$
\gamma_{-\rho}\circ \pi(C_{V})=\sum_{\mu\in \Pi(V)}m_{V}(\mu)K_{2\mu}$
as desired.
\end{proof}

Recall that the character of $V$ is defined by $\chi(V)=\sum_{\mu\in \Pi(V)} m_{V}(\mu) e^{\mu}$ \cite{Hum72, Jan96}. Hence $\gamma_{-\rho}\circ\pi (C_{V})$ is equal to the character $\chi(V)$ with $e^{\mu}$ replaced with $K_{2\mu}$, and we may study the centre $Z({\rm U}_q(\mathfrak{g}))$ from a representation-theoretic point of view. 

We define $\bar{R}({\rm U}_q(\mathfrak{g})):=\mathbb{C}(q)\otimes_{\mathbb{Z}} R({\rm U}_q(\mathfrak{g}))$, where $R({\rm U}_q(\mathfrak{g}))$ is the Grothendieck ring of the category of finite dimensional ${\rm U}_q(\mathfrak{g})$-modules of type $1$. It is well known that the isomorphism classes $[L(\lambda)], \lambda\in P^+$ form a basis for  $\bar{R}({\rm U}_q(\mathfrak{g}))$. Define $S({\rm U}_q(\mathfrak{g}))$ to be the subalgebra of $\bar{R}({\rm U}_q(\mathfrak{g}))$ which has a basis consisting of isomorphism classes $[L(\lambda)], \lambda\in M^+$. 


Note that for any $\lambda,\mu\in M^+$ there is a unique decomposition $L(\lambda)\otimes L(\mu)\cong \bigoplus_{\nu}c_{\lambda,\mu}^{\nu}L(\nu)$, where $\nu\in M^+$ and $c_{\lambda,\mu}^{\nu}$ is the multiplicity of $L(\nu)$. Therefore, given any isomorphism class $[V]\in S({\rm U}_q(\mathfrak{g}))$, we have $\Pi(V)\subseteq M$ and hence can define a corresponding central element $C_{V}$. This gives rise to the following commutative diagram:
\begin{equation}\label{eq: commdia}
  \begin{tikzcd}
S({\rm U}_q(\mathfrak{g}))\arrow[rd,"\xi"] \arrow[d]& \\
  Z({\rm U}_q(\mathfrak{g}))  \arrow[r, "\gamma_{-\rho}\circ\pi "] &  ({\rm U}_{{\rm ev}}^0)^W. 
\end{tikzcd}
\end{equation}

\begin{lemma}\label{lem: SU}
   There is an algebra isomorphism $\xi: S({\rm U}_q(\mathfrak{g}))\rightarrow ({\rm U}_{{\rm ev}}^0)^W$ defined by
   \[  \xi([V])= \sum_{\mu\in \Pi(V)} m_{V}(\mu) K_{2\mu},\]
     where $m_{V}(\mu)= {\rm dim} L(\lambda)_{\mu}$.
\end{lemma}
\begin{proof}
   First, since $\xi([V][W])=\xi([V\otimes W])=\xi([V])\xi([W])$ (see, e.g., \cite[\S 22.5, Proposition B]{Hum72}) $\xi$ is an algebra homomorphism. If $\xi([V])=\xi([W])$, then $V$ and $W$ have the same character and hence $V\cong W$. Therefore $\xi$ is injective. It remains to show the surjectivity.

  Recall from \lemref{lem: basisUW} that $({\rm U}_{{\rm ev}}^0)^W$  has a basis ${\rm av}(\lambda)=\sum_{w\in W}K_{2w\lambda}, \lambda\in M^+$. It suffices to show that for any $\lambda\in M^+$ there exists a ${\rm U}_q(\mathfrak{g})$-module $V$ such that $\xi([V])={\rm av}(\lambda)$. We use induction on $\lambda$. 

  For the base case, if $\lambda\in M^+$ and there are no dominant weights lower than $\lambda$, we have
  \[ \sum_{w\in W} K_{2w\lambda}= \frac{|W|}{|W\lambda|} \sum_{\mu\in W\lambda} K_{2\mu}= \frac{|W|}{|W\lambda|} \xi([L(\lambda)]), \]
  where $W\lambda:=\{w\lambda\mid w\in W\}$ denotes the  $W$-orbit of $\lambda$. For general $\lambda\in M^+$, we have 
  \[
    \begin{aligned}
      \sum_{w\in W} K_{2w\lambda}&= \frac{|W|}{|W\lambda|}(\xi([L(\lambda)])- \sum_{\mu\in \Pi(\lambda)\backslash W\lambda} {\rm dim}\, L(\lambda)_{\mu} K_{2\mu} )\\
        &= \frac{|W|}{|W\lambda|}\xi([L(\lambda)])- \sum_{\mu<\lambda,\, \mu\in P^+} \frac{|W\mu|}{|W\lambda|} {\rm dim}\, L(\lambda)_{\mu}\sum_{w\in W}K_{2
        w\mu},
    \end{aligned}
  \]
  where $\mu<\lambda$ and $\mu\in P^+$ imply that $\mu\in M^+$. By induction hypothesis, each sum $\sum_{w\in W}K_{2w\mu}$ has a preimage  in $S({\rm U}_q(\mathfrak{g}))$ and so does $ {\rm av}(\lambda)=\sum_{w\in W} K_{2w\lambda}$. This completes the proof.
\end{proof}

The following is a consequence of the commutative diagram \eqref{eq: commdia} and \lemref{lem: SU}.
\begin{corollary}\label{coro: SUiso}
  The algebra   $S({\rm U}_q(\mathfrak{g}))$ is isomorphic to $ Z({\rm U}_q(\mathfrak{g}))$, with each isomorphism class $[V]\in S({\rm U}_q(\mathfrak{g}))$ assigned to the central element $C_{V}$.
\end{corollary}


\subsection{Proof of the main theorem}
Recall from \secref{sec: monoidalg} the monoid algebra $\mathbb{C}[M^+]$ generated by $X^{\lambda}, \lambda\in {\rm Hilb}(M^+)$. The explicit relations among these generators are given in \thmref{thm: monoalg1} and \thmref{thm: monoalg2} for the Lie algebra $\mathfrak{g}$ of type $\RN{1}$ and $\RN{2}$, respectively. In the sequel, we shall consider the monoid algebra $\mathbb{C}(q)[M^+]=\mathbb{C}(q)\otimes_{\mathbb{C}} \mathbb{C}[M^+]$, of which the generators and relations remain the same.

\begin{lemma}\label{lem: monoalgiso}
 The monoid algebra $\mathbb{C}(q)[M^+]$ is isomorphic to $S({\rm U}_q(\mathfrak{g}))$, with each generator $X^{\lambda}$ mapped to the isomorphism class $[T(\lambda)]$ for $\lambda\in {\rm Hilb}(M^+)$.
\end{lemma}
\begin{proof}
  Recall that as a vector space $\mathbb{C}(q)[M^+]$ has a basis $X^{\lambda}$ with $\lambda\in M^+$. If  $S({\rm U}_q(\mathfrak{g}))$ has a basis $[T(\lambda)], \lambda\in M^+$ then there exists a bijective linear map  sending $X^{\lambda}$ to $[T(\lambda)]$ for all $\lambda\in M^+$. Moreover, this linear map preserves the algebra structure since $[T(\lambda)][T(\mu)]=[T(\lambda)\otimes T(\mu)]=[T(\lambda+\mu)]$.

  Now it suffices to show that $[T(\lambda)], \lambda\in M^+$ make up a basis for $S({\rm U}_q(\mathfrak{g}))$. Note that there is the following decomposition 
  \begin{equation}\label{eq: Tla}
    T(\lambda)\cong L(\lambda) \oplus \bigoplus_{\mu<\lambda,\,  \mu\in M^+} L(\mu)^{\oplus m_{\lambda,\mu}}.
  \end{equation}
  Recall that  by definition  $[L(\lambda)], \lambda\in M^+$ form a basis for $S({\rm U}_q(\mathfrak{g}))$. For any $\gamma\in M^+$, the set $I(\gamma)=\{\mu\in M^+\mid \mu\leq \gamma \}$ is finite. Then the elements $[T(\lambda)], \lambda\in I(\gamma)$ are linearly independent, since we may arrange $I(\gamma)$ non-decreasingly under the partial order and the transformation matrix between $[T(\lambda)], \lambda\in I(\gamma)$ and $[L(\lambda)], \lambda\in I(\gamma)$ is non-singular by \eqref{eq: Tla}. Moreover, $[T(\lambda)], \lambda\in M^+$ is a spanning set of $S({\rm U}_q(\mathfrak{g}))$, since each element of $S({\rm U}_q(\mathfrak{g}))$ is a finite linear combination of $[L(\lambda)], \lambda\in M^+$ and hence a finite linear combination of $[T(\lambda)], \lambda\in M^+$ by using \eqref{eq: Tla}. Therefore, $[T(\lambda)], \lambda\in M^+$ form a basis for $S({\rm U}_q(\mathfrak{g}))$. 
\end{proof}

\begin{corollary}\label{coro: gen}
 We have the following: 
 \begin{enumerate}
   \item  The isomorphism classes $[T(\lambda)], \lambda\in {\rm Hilb}(M^+)$ generate the algebra  $S({\rm U}_q(\mathfrak{g}))$.
   \item The  elements $C_{T(\lambda)}, \lambda \in {\rm Hilb}(M^+)$ generate the centre $Z({\rm U}_q(\mathfrak{g}))$.
 \end{enumerate}
\end{corollary}
\begin{proof}
Part (1) is a consequence of  \lemref{lem: monoalgiso}, and part (2) follows from the combination of part (1) and \corref{coro: SUiso}.
\end{proof}




We are in a position to prove our main theorem.

\begin{proof}[Proof of \thmref{thm: cen}]
 By \corref{coro: gen}, $Z({\rm U}_q(\mathfrak{g}))$ is generated by $C_{T(\lambda)}$ for all $ \lambda\in {\rm Hilb}(M^+)$. Combining \corref{coro: SUiso} and \lemref{lem: monoalgiso}, we have the algebra isomorphism $Z({\rm U}_q(\mathfrak{g}))\cong \mathbb{C}(q)[M^+]$,  with each generator $C_{T(\lambda)}$  assigned to $X^{\lambda}$ for $\lambda\in {\rm Hilb}(M^+)$. Now the theorem follows from \thmref{thm: monoalg1} and \thmref{thm: monoalg2} by replacing the ground field $\mathbb{C}$ with $\mathbb{C}(q)$.
\end{proof}

\begin{remark}
 In part (2) of \thmref{thm: cen}, one can show that $C_{L(\lambda)}, \lambda\in {\rm Hilb}(M^+)$ also form a generating set of $Z({\rm U}_q(\mathfrak{g}))$, but they do not obey the same relations. For examples of these relations, refer to \exref{exam: rel}, where $x_{\lambda}$ should be replaced with $C_{T(\lambda)}$ for $\lambda\in {\rm Hilb}(M^+)$. 
\end{remark}

\appendix

\section{Proofs of commutative relations}\label{sec: proofs}

In this appendix, we shall prove commutative relations in  \lemref{lem: Gamma} and  \propref{prop: commrel}. First we prove \lemref{lem: Gamma}.
\begin{proof}[Proof of \lemref{lem: Gamma}]
 For succinctness, we just prove that $C_V^{(k)}$ commutes with $E_i$ and the other cases can be treated similarly. Since $[\Gamma_V,\Delta(x)]=0$, then we have $[(\Gamma_V)^k,\Delta(x)]=0$ for any $x\in\rm U$.
Assuming that $(\Gamma_V)^k=\sum_{j}A_j\otimes B_j$, we have 
\[
\begin{aligned}
  0=& {\rm Tr}_1((K_i^{-1}K_{2\rho}\otimes1) [(\Gamma_V)^k, \Delta(E_i)] )\\
    =& {\rm Tr}_1(\sum_j(K_i^{-1}K_{2\rho}\otimes1)[A_j\otimes B_j, K_i\otimes E_i+ E_i\otimes 1] )\\
    =&{\rm Tr}_1( \sum_j(K_i^{-1}K_{2\rho}\otimes1)(A_jK_i\otimes B_jE_i+ A_jE_i\otimes B_j\\
    & - K_iA_j\otimes E_iB_j- E_iA_j\otimes B_j ) ).
\end{aligned}
\]
This can be written as a sum of two terms: the first term is 
\[ 
\begin{aligned}
&{\rm Tr}_1( \sum_j(K_i^{-1}K_{2\rho}\otimes1)(A_jK_i\otimes B_jE_i-K_iA_j\otimes E_iB_j) )\\
=&\sum_j{\rm Tr}(K_{2\rho}A_j)(B_jE_i- E_iB_j)\\
=&[C_V, E_i],
\end{aligned}
  \]
and the second term is 
\[
\begin{aligned}
&{\rm Tr}_1( \sum_j(K_i^{-1}K_{2\rho}\otimes1)(A_jE_i\otimes B_i-E_iA_j\otimes B_j) )\\
=&\sum_j({\rm Tr}(K_i^{-1}K_{2\rho}A_jE_i)-{\rm Tr}(K_i^{-1}K_{2\rho}E_iA_j) )B_j,
\end{aligned}
\]
which is equal to $0$ since 
$$
\begin{aligned}
&{\rm Tr}(K_i^{-1}K_{2\rho}A_jE_i)={\rm Tr}(E_iK_i^{-1}K_{2\rho}A_j)
\\
=&q^{-(2\rho-\alpha_i,\alpha_i)}{\rm Tr}(K_i^{-1}K_{2\rho}E_iA_j)={\rm Tr}(K_i^{-1}K_{2\rho}E_iA_j).
\end{aligned}
$$
Therefore, we have $[C_V^{(k)},E_i]=0$.
\end{proof}

Now we turn to prove \propref{prop: commrel}. Let us start with the following lemma. 
\begin{lemma}\label{lem: Krel}
 Let $\zeta=\zeta_V: {\rm U}_q(\mathfrak{g})\rightarrow {\rm GL}(V)$ be the linear representation associated to $V$, and let $\mathcal{K}_V$ be as defined in \eqref{eq: KV}. 
  \begin{enumerate}
    \item For $1\leq i,j\leq n$, we have 
        \begin{equation*}
          \begin{aligned}
            \mathcal{K}_V(\zeta(K_i^{\pm})\otimes K_j^{\pm 1}) &=(\zeta(K_i^{\pm})\otimes K_j^{\pm 1})\mathcal{K}_V\\
           \mathcal{K}_V(\zeta(E_i)\otimes 1)&=(\zeta(E_i)\otimes K^2_{i})\mathcal{K}_V,\\
            \mathcal{K}_V(1\otimes E_i)&=(\zeta(K_i^2)\otimes E_i)\mathcal{K}_V,\\
           \mathcal{K}_V(\zeta(F_i)\otimes 1)&=(\zeta(F_i)\otimes K_{i}^{-2})\mathcal{K}_V,\\
            \mathcal{K}_V(1\otimes F_i)&=(\zeta(K_i^{-2})\otimes F_i)\mathcal{K}_V.
          \end{aligned}   
        \end{equation*}
     \item For any $x\in {\rm U}_q(\mathfrak{g})$, we have 
     \[  \mathcal{K}_{V}\phi^2(\Delta(x))=\Delta(x)\mathcal{K}_{V},
\]
where the first tensor factors in both $\Delta(x)$ and $\phi^2(\Delta(x))$ are regarded as elements in ${\rm End}(V)$ via the linear representation $\zeta$.
  \end{enumerate}

\end{lemma}
\begin{proof}
  For part (1), we only prove the second equation; the others can be treated similarly. Recall that $P_{\eta}: V\rightarrow V_{\eta}$ is the linear projection onto the weight space $V_{\eta}$ of $V$. It can be verified easily that 
  \[
\begin{aligned}
   P_{\eta}\zeta(K_{i}^{\pm 1}) &= \zeta(K_i^{\pm 1}) P_{\eta},\\
   P_{\eta} \zeta(E_i)&= \zeta(E_i) P_{\eta-\alpha_i},\\
   P_{\eta} \zeta(F_i)&= \zeta(F_i) P_{\eta+ \alpha_i},
\end{aligned}
  \]
where $P_{\eta-\alpha_i}:=0$ (resp. $P_{\eta+\alpha_i}:=0$) if $\eta-\alpha_i\not\in\Pi(V)$ (resp. $\eta+\alpha_i\not\in\Pi(V)$).  Using the second equality, we have 
\begin{align*}
\mathcal{K}_V(\zeta(E_i)\otimes 1)&=\sum_{\eta\in\Pi(V)}P_{\eta}^V\zeta(E_i)\otimes K_{2\eta}\\
&=\sum_{\eta\in\Pi(V)}\zeta(E_i)P_{\eta-\alpha_i}\otimes K_{2\eta}\\
&=(\zeta_{V}(E_i)\otimes K_i^2)\mathcal{K}_V.
\end{align*}

Part (2) is a direct consequence of part (1), and we take $x=E_i$ as an example. Using $\phi^2(\Delta(E_i))=K^{-1}_i\otimes E_i+E_i\otimes K_{i}^{-2}$, we have
\begin{align*}
\mathcal{K}_V\phi^2(\Delta(E_i))&=\mathcal{K}_V(\zeta(K^{-1}_i)\otimes E_i+\zeta(E_i)\otimes K_i^{-2})\\
&=(\zeta(K_i)\otimes E_i+\zeta(E_i)\otimes 1)\mathcal{K}_V\\
&=\Delta(x)\mathcal{K}_V.
\end{align*}
This completes the proof.
\end{proof}

Now we are in a position to prove \propref{prop: commrel}.

\begin{proof}[Proof of \propref{prop: commrel}]
By \eqref{eq: quasiR} there are equations $\mathcal{R}_{V}\Delta(x)=\phi(\Delta'(x))\mathcal{R}_V$ and $\mathcal{R}^T\Delta'(x)=\phi(\Delta(x))\mathcal{R}^T$. Applying the algebra homomorphism $\phi$ to the latter, we have  $\phi(\mathcal{R}^T\Delta'(x))=\phi^2(\Delta(x))\phi(\mathcal{R}^T_V)$. Then it follows that 
 \begin{align*}
  \Gamma_V\Delta(x)&=\mathcal{K}_{V}\phi(\mathcal{R}^T_V)\mathcal{R}_{V}\Delta(x)\\
  &=\mathcal{K}_{V}\phi(\mathcal{R}^T_V)\phi(\Delta'(x))\mathcal{R}_V\\
&=\mathcal{K}_{V}\phi^2(\Delta(x))\phi(\mathcal{R}^T_V)\mathcal{R}_V\\
&=\Delta(x)\Gamma_{V},
\end{align*}
where the last equation follows from part (2) of \lemref{lem: Krel}.
\end{proof}






\end{document}